\documentclass{amsart}
\usepackage{tikz}
\usetikzlibrary{arrows,matrix}

\begin{document}

\title[Perfectoid spaces and arithmetic jets]{Perfectoid spaces arising from\\
 arithmetic jet spaces}
\bigskip

\def \h{\hat{\ }}
\def \cO{\mathcal O}
\def \ra{\rightarrow}
\def \bZ{{\mathbb Z}}
\def \cP{\mathcal V}
\def \cH{{\mathcal H}}
\def \cB{{\mathcal B}}
\def \d{\delta}
\def \cC{{\mathcal C}}
\def \Z{\mathbb Z}
\def \varphic{f^{\flat}}
\def \h{\widehat{\ }}

\newcommand{\Spec}{\operatorname{Spec}}
\newcommand{\Spa}{\operatorname{Spa}}
\newcommand{\Spf}{\operatorname{Spf}}

\newcommand{\Hom}{\operatorname{Hom}}

\newtheorem{THM}{{\!}}[section]
\newtheorem{THMX}{{\!}}
\renewcommand{\theTHMX}{}
\newtheorem{theorem}{Theorem}[section]
\newtheorem{corollary}[theorem]{Corollary}
\newtheorem{lemma}[theorem]{Lemma}
\newtheorem{proposition}[theorem]{Proposition}
\newtheorem{thm}[theorem]{Theorem}
\theoremstyle{definition}
\newtheorem{definition}[theorem]{Definition}
\theoremstyle{remark}
\newtheorem{remark}[theorem]{Remark}
\newtheorem{example}[theorem]{\bf Example}
\numberwithin{equation}{section}
\address{\ } 

\author{Alexandru Buium}
\address{Department of Mathematics and Statistics,
University of New Mexico, Albuquerque, NM 87131, USA}
\email{buium@math.unm.edu} 

\author{Lance Edward Miller}
\address{Department of Mathematical Sciences,  309 SCEN,
University of Arkansas, 
Fayetteville, AR 72701}
\email{lem016@uark.edu}

\maketitle
\begin{abstract}
Using  arithmetic jet spaces  \cite{char} we attach perfectoid spaces 
\cite{scholze} to smooth schemes and we attach morphisms of perfectoid spaces to  $\d$-morphisms \cite{equations} of smooth schemes. We also study  perfectoid spaces attached  to  arithmetic differential equations \cite{equations}
defined by some of the remarkable  $\d$-morphisms appearing in the theory  such as the $\d$-characters  of elliptic curves  \cite{char} and the $\d$-period maps on modular curves  \cite{equations}.
\end{abstract}

\section{Introduction}

Let $R$ be a 
 complete discrete valuation ring  with maximal ideal generated by an odd prime $p$ and  with an algebraically closed residue field $k=R/pR$. Consider the discrete valued field $K=R[1/p]$ and 
 the metric completion ${\mathbb K}$ of $K[p^{1/p^{\infty}}]$ which is a perfectoid field;  cf. \cite{bhatt, scholze}. Set ${\mathbb K}^{\circ}$ and ${\mathbb K}^{\circ \circ}$ be the valuation ring of ${\mathbb K}$ and the maximal ideal of the valuation ring respectively. The goal of the present paper 
is to use arithmetic jet spaces introduced in \cite{char, pjets} as a tool to construct perfectoid spaces over ${\mathbb K}$ 
, in the sense of \cite{scholze},  attached to various geometric objects over $R$. 

Our first move will be  to construct a functor 
$$
P^{\infty}:{\mathcal S}\ra {\mathcal P}$$
from the category 
 ${\mathcal S}$  of smooth quasi-projective schemes over $R$ to the category
${\mathcal P}$  of perfectoid spaces  over ${\mathbb K}$. 
Our functor $P^{\infty}$  will be 
unique up to an isomorphism of functors.

Explicitly, 
recall from \cite{char, pjets} that for any scheme of finite type $X$ over $R$ one can attach 
its arithmetic jet spaces
$J^n(X)$, $n\geq 0$, which are $p$-adic formal schemes with $J^0(X)=\widehat{X}$. Here and later $\widehat{\ }$ means $p$-adic completion.
Then for any smooth affine scheme $X=\text{Spec}\ B$ over $R$  we let  $J^{\infty}(B)$ be the 
{\it infinite arithmetic jet algebra}, defined as the $p$-adic
completion of the direct limit of the rings $\cO(J^n(X))$; the ring $J^{\infty}(B)$
 carries a canonical {\it Frobenius lift} $\phi$.
Our functor $P^{\infty}$ will attach to any such $X$  the perfectoid 
space $\Spa(\widehat{\mathcal B}[1/p],
\widehat{\mathcal B})$ where 
 ${\mathcal B}$ is the direct limit of  the ring $(J^{\infty}(B) \otimes_R{\mathbb K}^\circ)^{\widehat{\ }}$  along  its induced {\it  relative Frobenius lift} $\Phi$ (cf. the body of the paper for the latter concept); the ${\mathbb K}$-algebra $\widehat{\mathcal B}[1/p]$ turns out to be canonically a perfectoid algebra.  In particular $\text{Spec}\ R$ is sent, by the above functor, into $\Spa({\mathbb K},{\mathbb K}^\circ)$. The case when $X$ is not necessarily affine  is treated by gluing the affine pieces using a ``yoga" of {\it principal covers}.
Cf. Section  2 for a review of these concepts and Section 3 for the construction of the functor $P^{\infty}$.

 As a matter of fact we will show, in Section 3, that the functor $P^{\infty}$ extends to a functor, which we abusively still denote by $P^{\infty}$, that fits into a diagram
 $$\begin{array}{rcl}
 {\mathcal S}_{\d} & \stackrel{P^{\infty}}{\ra} &  {\mathcal P}_{\Phi}\\
 \text{incl} \uparrow & \ & \downarrow \text{forget} \\
 {\mathcal S} & \stackrel{P^{\infty}}{\ra} &  {\mathcal P} 
 \end{array}
 $$
 where the categories above are described as follows. The category ${\mathcal S}_{\d}$ is the category whose objects are the same as those of ${\mathcal S}$ but whose morphisms are the $\d$-{\it morphisms} of schemes. Recall, a $\d$-{\it morphism} of order $n$ between two smooth $R$-schemes of finite type $X$ and $Y$ is, by definition, a morphism of formal schemes $J^n(X)\ra \widehat{Y}$;
 $\d$-morphisms of various orders can be composed and they yield a category. 
 The category ${\mathcal P}_{\Phi}$ is the category whose objects are  perfectoid spaces over ${\mathbb K}$ equipped with what we shall call a {\it principal cover} plus a {\it relative Frobenius lift}, and whose morphisms are compatible with these data.  The functor ``$\text{incl}$" is the identity on objects and the ``inclusion" on morphisms. The functor ``$\text{forget}$" is the forgetful functor.
 
Let ${\mathcal S}_k$ be the category of $k$-schemes. We will prove the following, cf. Corollary \ref{pong}.
 
 \begin{theorem}\label{soare}
 We have a commutative diagram of functors:
 $$
  \begin{array}{rcccl}
 {\mathcal S}_{\d}  & \stackrel{P^{\infty}}{\ra}   & {\mathcal P}_{\Phi}  & \ & \ \\
 \textrm{incl} \uparrow & \ & \ & \stackrel{\textrm{mod}}{\searrow} & \ \\
 {\mathcal S} & \stackrel{\textrm{Green}}{\ra} &  {\mathcal S}_k & \stackrel{\textrm{perf}}{\ra} & {\mathcal S}_k \\
 \end{array}
 $$
 \end{theorem}
 
 The functor
 ``$\text{mod}$",
 ${\mathbb X}\ra \bmod({\mathbb X})=:{\mathbb X}_k$ will be referred to as the {\it reduction modulo ${\mathbb K}^{\circ \circ}$ functor}. The functor ``$\text{Green}$"  is the {\it Greenberg transform functor} \cite{pjets, greenberg}. The functor  ``$\text{perf}$" is the {\it perfection functor} $V\mapsto V_{\text{perf}}$.

 Some remarkable $\d$-morphisms were constructed in  \cite{char, difmod, equations}. Among them are the {\it $\d$-characters} (differential characters)
 \begin{equation}
 \psi:J^2(E)\ra \widehat{{\mathbb G}_a}\end{equation}
  of Abelian schemes $E$ \cite{char}, and the  {\it $\d$-period maps}  
  \begin{equation}
  \wp:J^2(X)\ra \widehat{{\mathbb A}^2},\end{equation}
   on certain  open sets $X$ of modular curves, where 
 the components of $\wp$ are defined by {\it isogeny covariant differential modular forms} \cite{difmod, equations}. 
 For simplicity we will only be interested here in the case $E$ has dimension $1$ (an elliptic curve) and the modular curve has level $1$ (the $j$-line). 
 The maps $\psi$ and $\wp$ will be reviewed in the body of the paper.
 With these maps at our disposal, we have induced morphisms of perfectoid spaces,
\begin{equation}\label{Ppsi}
P^{\infty}(\psi):P^{\infty}(E)\ra P^{\infty}({\mathbb G}_a),\end{equation}
\begin{equation}\label{Pwp}
P^{\infty}(\wp):P^{\infty}(X)\ra P^{\infty}({\mathbb A}^2),\end{equation}
which we may still refer to as the perfectoid {\it $\d$-characters} and the perfectoid {\it $\d$-period maps}
 respectively. One can also consider the {\it projectivized $\d$-period map},
 \begin{equation}
 \tilde{\wp}:J^2(X)\backslash \wp^{-1}(0)\stackrel{\wp}{\ra} \widehat{{\mathbb A}^2}\backslash \{0\}\stackrel{\text{can}}{\ra} \widehat{{\mathbb P}^1}.
 \end{equation}
 The map $\tilde{\wp}$ might  be more natural than $\wp$ in the sense that its closer analogy with classical period maps but we prefer to work with $\wp$ instead  because $\tilde{\wp}$ is not a $\d$-morphism; it has singularities along $\wp^{-1}(0)$ so, in our formalism, we cannot attach to it a map of perfectoid spaces. One can do that if one appropriately generalizes our formalism but such a generalization would take us too far afield.
 
 \

 There is an analogue 
 $$\psi:J^1({\mathbb G}_m)\ra \widehat{{\mathbb G}_a}$$
 of the $\psi$ above called the canonical {\it  $\d$-character of ${\mathbb G}_m$} 
 which unlike the case of abelian schemes is entirely explicit. Thus, there is an induced map
  \begin{equation}\label{Popsi}
P^{\infty}(\psi):P^{\infty}({\mathbb G}_m)\ra P^{\infty}({\mathbb G}_a),\end{equation}
which we may still refer to as the canonical perfectoid {\it  $\d$-character}.

Section 4  begins with  a concrete description of the endomorphism $P^{\infty}([p])$ of $P^{\infty}({\mathbb G}_m)$ induced by the $p$-power isogeny $[p]$ of ${\mathbb G}_m$; in particular we will prove the following, cf. Theorem \ref{poi}.

\begin{theorem}
The morphism $P^{\infty}([p]):P^{\infty}({\mathbb G}_m)\ra P^{\infty}({\mathbb G}_m)$ is a closed immersion but not an isomorphism.
\end{theorem}

We also construct of a canonical  section  for the canonical
 $\d$-character  $P^{\infty}(\psi)$ in \ref{Popsi}. More precisely, via this we introduce an entirely new concept in $\d$-geometry, namely {\it  cocharacters} of ${\mathbb G}_m$ which are certain  morphisms
\begin{equation}
\label{eva}
\sigma:P^{\infty}({\mathbb G}_a)\ra P^{\infty}({\mathbb G}_m)\end{equation} 
that play the role of $1$-parameter subgroups. We will prove the following relationship for the canonical character. 

\begin{theorem}
There is a unique cocharacter $\sigma$ with the property that $\sigma$ is a right inverse of the canonical perfectoid $\d$-character $P^{\infty}(\psi)$. 
\end{theorem} 

Section 5 begins with a review of $\d$-characters of elliptic curves (cf. \cite{char, equations}). Here again, we will discuss  a concept of {\it  cocharacters} which are certain morphisms
\begin{equation}
\label{adam}
\sigma:P^{\infty}({\mathbb G}_a)\ra P^{\infty}(E)\end{equation}
that play, as before, the role of $1$-parameter subgroups. To state our result 
  below recall that an elliptic curve over $R$ is called 
 {\it ordinary} respectively {\it supersingular} if its reduction modulo $p$ is so. An elliptic curve over $R$ will be  called
 {\it superordinary} if it has ordinary reduction and its Serre-Tate parameter  is not congruent to $1 \bmod p^2$. Recall that an ordinary elliptic curve has Serre-Tate parameter equal to $1$ if and only if the elliptic curve has a Frobenius lift, equivalently if it is a canonical lift. 
We will prove the following result; cf. Theorem \ref{sandoval}.

\begin{theorem}\label{ruru}
Assume $E$ has no Frobenius lift and let $P^{\infty}(\psi)$ be the perfectoid $\d$-character described in \ref{Ppsi}. 

1) If $E$ is supersingular there exists a unique  cocharacter $\sigma:P^{\infty}({\mathbb G}_a)\ra P^{\infty}(E)$
with the property that $\sigma$ is  a right inverse to $P^{\infty}(\psi)$.

2) If $E$ is superordinary or ordinary and defined over $\bZ_p$ then there exists no  cocharacter $\sigma:P^{\infty}({\mathbb G}_a)\ra P^{\infty}(E)$ with the property that $\sigma$ is  a right inverse to $P^{\infty}(\psi)$.
\end{theorem}

Our next aim, in Section 6-8, will be to analyze  perfectoid spaces attached to some remarkable {\it arithmetic differential equations}, in the sense of \cite{equations}. 
Here an arithmetic differential equation of order $n$ on a scheme $X$ is a closed formal subscheme 
of  $J^n(X)$. The examples we are interested in here are 
 the {\it arithmetic analogues of Manin kernels}, which are the kernels of the $\d$-characters
 $\psi$ above  and  the {\it $\d$-isogeny classes}. These are defined as 
 preimages via the $\d$-period map $\wp$ above of lines in ${\mathbb A}^2$ passing through the origin. These arithmetic differential equations played a key role in \cite{local}
where an application to Heegner points was considered and in \cite{BYM} where an arithmetic analogue of Painlev\'{e} VI equations was introduced and studied. We refer to Sections 6-8 for a review of these concepts and for statements/proofs of our results. The key step in our construction will be to show that the said arithmetic differential equations enjoy a remarkable property which we call {\it quasi-linearity}. We will undertake a detailed analysis of quasi-linear arithmetic differential equations; as a corollary of our analysis we will show that certain algebras attached to them are perfectoid. With the definitions reviewed in the body of the paper we will prove the following Theorems \ref{hoi} and \ref{boi} below; cf. Corollaries \ref{lll} and \ref{ppp},
respectively.

\begin{theorem}\label{hoi} 
Let $E$ be a superordinary elliptic curve over $R$.
There exists 
a perfectoid space $P^{\sharp}(E)$ in ${\mathcal P_{\Phi}}$  plus a closed immersion $P^{\sharp}(E)\ra P^{\infty}(E)$ in ${\mathcal P_{\Phi}}$ such that the following properties hold:

\begin{enumerate}

\item $P^{\sharp}(E)_{\text{red}}$ is the perfection of a profinite pro-\'{e}tale cover of the reduction modulo $p$ of $E$. 
 
\item For any torsion point $T:\Spec R \ra E$, the induced morphism on perfectoid spaces $P^{\infty}(T):P^{\infty}(\Spec R)\ra P^{\infty}(E)$ factors through $P^{\sharp}(E)\ra P^{\infty}(E)$.

\item The  attachment $E\mapsto P^{\sharp}(E)$ is functorial in $E$, with respect to isogenies of degree prime to $p$, and also with respect to translations by torsion points.

\end{enumerate} 
  
\end{theorem}

The link between $P^{\sharp}(E)$ in Theorem \ref{hoi} and the map $P^{\infty}(\psi)$ in \ref{Ppsi} is provided by the fact that
the composition of $P^{\sharp}(E)\ra P^{\infty}(E)$ with the map $P^{\infty}(\psi)$ factors through the origin $0:P^{\infty}(\Spec R)\ra
P^{\infty}({\mathbb G}_a)$.
We actually expect that $P^{\sharp}(E)$ coincides with the fiber of $P^{\infty}(\psi)$ at $0$.

\begin{theorem} \label{boi}
Let $X$ be the $j$-line ${\mathbb A}^1$    over $R$ from which one removes the points $j=0,1728$ and the supersingular locus and let $Q:\text{Spec}\ R\ra X$ be a point corresponding to a superordinary elliptic curve $E$.
There exists 
a perfectoid space $P^Q(X)$ in ${\mathcal P}_{\Phi}$ 
plus a closed immersion $P^Q(X)\ra P^{\infty}(X)$ 
in ${\mathcal P}_{\Phi}$
such that the following properties hold.
 
\begin{enumerate}

\item The reduced locus $P^Q(X)_{\text{red}}$  equals  the perfection of a profinite pro-\'{e}tale cover of the reduction modulo $p$ of $X$. 

\item The map $P^{\infty}(Q):P^{\infty}(\Spec R)\ra P^{\infty}(X)$ factors through the described map $P^Q(X)\ra P^{\infty}(X)$.

\item For any point $Q':\Spec R\ra X$ corresponding to an elliptic curve $E'$ that has an isogeny of prime to $p$ degree to $E$, the curve $E'$ is superordinary and $P^Q(X)= P^{Q'}(X)$ in $P^{\infty}(X)$.

\end{enumerate}

 \end{theorem}

The link between $P^Q(X)$ in Theorem \ref{boi} and the map $P^{\infty}(\wp)$ in Equation~\ref{Pwp} is provided by the fact that
the composition of  $P^Q(X)\ra P^{\infty}(X)$ with $P^{\infty}(\wp)$ factors through an embedding $P^{\infty}(L)\ra P^{\infty}({\mathbb A}^2)$ where $L\subset {\mathbb A}^2$ is a line passing through the origin.
We actually expect that $P^Q(X)$ is the pull-back of $L$ via $P^{\infty}(\wp)$.

\begin{remark}
Note that the infinite arithmetic jet algebra $J^{\infty}(B)$ possesses a {\it Fermat quotient operator} $\d$ defined by $\d x=(\phi(x)-x^p)/p$, where $\phi$ is the canonical  Frobenius lift on $J^{\infty}(B)$. The operator $\d$ was viewed in \cite{char, equations} as an analogue of a derivation (there it was called a $p$-{\it derivation}) and a substantial body of the classical geometric theory of differential equations turned out to have an arithmetic analogue in the setting of arithmetic jet spaces. On the other hand the relative Frobenius lift $\Phi$ on 
$(J^{\infty}(B) \otimes_R{\mathbb K}^\circ)^{\widehat{\ }}$ is constructed from $\phi$ and hence can be viewed as a sort of substitute for a derivative. 
Taking the direct limit of the latter ring along $\Phi$ intuitively amounts to adjoining primitives
(or more generally ``adjoining pseudo-differential operators") 
 to the infinite arithmetic jet algebra and hence, in particular,  allowing a type of integration in addition to differentiation. It would be interesting to understand completely the arithmetic analogues of the analytic/geometric aspects these primitives/integration/pseudo-differential operators possess. One such aspect 
is brought forward by Theorem \ref{ruru} and its analogue for ${\mathbb G}_m$, cf. Theorems \ref{sandocal},  which can be interpreted as saying that  certain arithmetic differential equations defined by $\d$-characters can/cannot be ``integrated" using ``arithmetic analogues of pseudo-differential operators." For a comparison of our theory with usual differential equations in the   complex analytic case we refer to subsections \ref{complexGm} and \ref{complexelliptic}. \end{remark}

\bigskip\noindent
{\bf Acknowledgment.}
The first author 
is grateful to Max Planck Institute for Mathematics in Bonn for its hospitality and financial support and to the Simons Foundation  for support through awards 311773 and
615356.

\section{Terminology and notation} 

We quickly review many of the constructions and notations we need. Throughout the paper $p$ will be an odd prime in $\bZ$ and rings are assumed to be commutative with identity. Unless otherwise stated, throughout the paper, we fix a complete discrete valuation ring $(R,(p),k)$ with maximal ideal generated by $p$ and algebraically closed residue field $k=\overline{k}=R/pR$. 

\

For any ring $S$ or Noetherian scheme $X$ we denote by  $\widehat{S}$ and $\widehat{X}$ the $p$-adic completions and we sometimes set $\overline{S}=S/pS$ and $\overline{X}:=X\otimes \bZ/p\bZ$. Recall that $\widehat{S}/p^m\widehat{S}=S/p^mS$. If $S$ is $p$-torsion free then so is $\widehat{S}$. We also have need to consider the total integral closure. For  $A$ a  subring of a ring $B$ one says that {\it $A$ is totally integrally closed in $B$} provided for any $b\in B$ when $b^{{\mathbb N}}$ is contained in a finitely generated $A$-submodule of $B$ then $b\in A$. Note when $A$ is Noetherian, then $A$ is totally integrally closed in $B$ if and only if $A$ is integrally closed in $B$.

\

\noindent {\bf Direct limits and Perfection:} Let $A$ be a ring equipped with a ring endomorphism $\Phi$. The direct limit $C$ of $A$ along $\Phi$ is the direct limit  of the system 
$$C = \varinjlim_{x \mapsto \Phi(x)} = A \stackrel{\Phi}{\ra} A \stackrel{\Phi}{\ra} A  \stackrel{\Phi}{\ra}..., i.e.,$$

$$C:=\frac{A \times {\mathbb N}}{\sim},\ \ (a,i)\sim(\Phi(a),i+1).$$
The  class of a pair $(a,i)$ in $C$ will be denoted by $[a,i]$.

\

The direct limit of most interest for us is the colimit perfection. For any ring $A$ of characteristic $p$ one defines its {\it perfection} via the direct limit along its Frobenius
$$A_{\text{perf}}:=\varinjlim_{x \mapsto x^p} A.$$
In particular, $A_{\text{perf}}$ is the unique perfect ring initial among perfect rings receiving maps from $A$. The natural Frobenius on $A_{\text{perf}}$ is then surjective. If $A$ is also reduced then 
Frobenius on $A_{\text{perf}}$ is bijective.

\

\noindent {\bf Smoothness:} Inherently, perfections need not be noetherian. By a {\it smooth} (respectively {\it \'{e}tale}) algebra we understand a  $0$-smooth (respectively $0$-\'{e}tale) 
algebra in the sense of \cite[ Ch. 10 ]{matsumura}, i.e.,  smooth or \'{e}tale algebras here are not necessarily finitely generated. Recall,  for finitely generated maps of Noetherian rings smooth is equivalent to flat, with smooth geometric fibers, see \cite[Thm. 28.10]{matsumura}. 
An $A$-algebra $B$  will be called here  {\it ind-\'{e}tale} if 
$B$ is the ascending union of a sequence of subrings $(B_n)_{n\geq 0}$ such that $B_0=A$ and 
  $B_{n+1}$ is an \'{e}tale finitely generated $B_n$-algebra for all $n\geq 0$. Note that 
any ind-\'{e}tale algebra is \'{e}tale. If $B$ is ind-\'{e}tale over $A$ then we say $\text{Spec}\ B\ra \text{Spec}\ A$ is  {\it pro-\'{e}tale}.

\

\noindent{\bf Frobenius lifts and $p$-derivations: }  For a more complete reference, \cite{char, Jo}. By the {\it Frobenius} on a ring of characteristic $p$ we mean the $p$-power Frobenius endomorphism. A {\it Frobenius lift} on a ring $A$ is a ring endomomorphism of $A$ whose reduction modulo $p$ is the Frobenius on $\overline{A}=A/pA$. More generally if $u:A\ra B$ is an $A$-algebra a {\it Frobenius lift} is a  ring homomorphism $A\ra B$ whose reduction modulo $p$ is the Frobenius on $\overline{A}$ composed with the map $\overline{u}:\overline{A}\ra \overline{B}$ induced by $u$. 
Consider the polynomial
$$C_p(X,Y):=\frac{X^p+Y^p-(X+Y)^p}{p} \in \bZ[X,Y]\, .$$
For an $A$-algebra $B$
 a  set theoretic map $\d:A \ra B$ is a   {\it $p$-derivation}  if $\d 1=0$ and, for all $a,b\in A$, 

\begin{enumerate}
\item $\d(a+b)  =  \d a+\d b + C_p(a,b)$,
\item $\d(ab)  =  a^p \d b+b^p \d a + p \cdot \d a \cdot \d b.$
\end{enumerate}

\noindent If $\d$ is a $p$-derivation, then the map
$\phi:A  \ra B$ defined by 
$\phi(a):= a^p+p\d a$, 
is a Frobenius lift.
Conversely, if $\phi:A \ra B$ is a Frobenius lift and  $B$ is $p$-torsion free 
then one can define a $p$-derivation $\d:A\ra B$ by
$\d a:=\frac{\phi(a)-a^p}{p}$. We say that $\d$ and $\phi$ are {\it attached}
to each other. 

Some refer to a ring $A$ equipped with a $p$-derivation $\d:A\ra A$ as a {\it $\delta$}-ring. As a category, $\delta$-rings is closed under limits and colimits. More specifically, since $\phi$ and $\d$ commute, the direct limit of $A$ along $\phi$ has a $p$-derivation induced by $\d$ and a Frobenius lift induced by $\phi$ which of course are attached to each other.

\

As $R$ is a complete discrete valuation ring $R$ with algebraically closed residue field, it is isomorphic to the $p$-typical ring of Witt vectors over $k$ and has a unique Frobenius lift $\phi$; moreover $\phi$ is bijective. We let $K=R[1/p]$ be the fraction field of $R$. Nearly all rings we consider will be $R$-algebras.

\subsection{Perfectoid algebras and spaces} We now review the basic theory of perfectoid algebras and spaces. For a much more detailed account, please see \cite{bhatt,scholze}. Let $(\pi_m)_{m\geq 0}$ be a sequence of elements in the algebraic closure of $K=R[1/p]$ such that $\pi_0=p$, $\pi_{m+1}^p=\pi_m$, and write
$\pi_m:=p^{1/p^m}$.
Let us consider the rings
$$R[p^{1/p^m}]\simeq R[z_m]/(z_m^{p^m}-p),\ \ R[p^{1/p^{\infty}}]=\bigcup R[p^{1/p^m}],\ \ 
z_m\mapsto z_{m+1}^p,$$
where $z_m$ are variables whose classes correspond to $p^{1/p^m}$.
The reduction modulo $p$ of the map $R[p^{1/p^m}]\ra R[p^{1/p^{m+1}}]$ is clearly injective.
Set ${\mathbb K}^\circ=(R[p^{1/p^{\infty}}])^{\widehat{\ }}$. This is a {\it non-discrete} valuation ring and we denote by ${\mathbb K}^{\circ \circ}$ its maximal ideal so ${\mathbb K}^{\circ \circ}$ is generated by $\{p^{1/p^n}\}$. The Frobenius lift of $R$ extends uniquely to a ring automorphism $\Phi$ of ${\mathbb K}^\circ$ 
that fixes all $p^{1/p^m}$. 

\

The reduction modulo $p$,  $\overline{\Phi}:{\mathbb K}^\circ/p{\mathbb K}^\circ \ra {\mathbb K}^\circ/p{\mathbb K}^\circ$ is then, of course,  bijective.
On the other hand    the Frobenius $ {\mathbb K}^\circ/p{\mathbb K}^\circ \ra {\mathbb K}^\circ/p{\mathbb K}^\circ$ 
is surjective but not injective; in particular,  $\Phi$ is not a Frobenius lift. We refer to $\Phi$ as the {\it canonical relative Frobenius lift} on ${\mathbb K}^\circ$. We sometimes refer to this simply as the relative Frobenius lift on ${\mathbb K}^\circ$. Note that ${\mathbb K}^\circ/{\mathbb K}^{\circ \circ}=R/pR=k$ and the automorphism of ${\mathbb K}^\circ/{\mathbb K}^{\circ \circ}$ induced by $\Phi$ is the Frobenius. The ring ${\mathbb K}:={\mathbb K}^\circ[1/p]$ is the metric completion of the ring $K[p^{1/p^{\infty}}]$ and is a perfectoid field in the sense of \cite[Def. 3.1.1]{bhatt}.

\begin{definition}
Assume we are given a ${\mathbb K}^\circ$-algebra ${\mathcal B}$. A {\it relative Frobenius lift} on ${\mathcal B}$ is a ring endomorphism of ${\mathcal B}$ which extends the canonical relative Frobenius lift on ${\mathbb K}^\circ$ and such that the induced endomorphism of ${\mathcal B}/{\mathbb K}^{\circ \circ}{\mathcal B}$ is the Frobenius. Call a relative Frobenius {\it invertible} if is it bijective, whence a ring isomorphism.  
\end{definition}

\noindent Note, the canonical relative Frobenius on ${\mathbb K}$ is invertible.

\

A {\it perfectoid ${\mathbb K}$-algebra} is a Banach ${\mathbb K}$-algebra $B$ whose ring $B^\circ$ of power bounded elements is bounded and for which the Frobenius on $B^\circ/pB^\circ$ is surjective, see \cite[Def. 6.2.1 part 1]{bhatt}. A {\it perfectoid ${\mathbb K}^\circ$-algebra} will mean here an ${\mathbb K}^\circ$-algebra ${\mathcal B}$ which is $p$-adically complete, $p$-torsion free, totally integrally closed in ${\mathcal B}[1/p]$, and such that Frobenius  on ${\mathcal B}/p{\mathcal B}$ is surjective. 

Recall, from \cite[Prop. 5.2.5, Def. 6.2.1]{bhatt} that $B:={\mathcal B}[1/p]$ has then a canonical Banach ${\mathbb K}$-algebra structure  and with this structure  $B$ is a perfectoid ${\mathbb K}$-algebra with $B^\circ={\mathcal B}$. Moreover, $B^\circ$ is open in $B$ and its induced topology is the $p$-adic topology.
Also, by loc. cit., the functors ${\mathcal B}\mapsto {\mathcal B}[1/p]$ and $B\mapsto B^\circ$ define an equivalence between the category with objects consisting of perfectoid ${\mathbb K}$-algebras where morphisms are continuous homomorphisms and the category of perfectoid ${\mathbb K}^\circ$-algebras.

 If $B$ is a perfectoid ${\mathbb K}$-algebra, a {\it relative Frobenius lift} is an endomorphism of $B$ sending $B^\circ$ into $B^\circ$, such that the induced endomorphism of the ${\mathbb K}^\circ$-algebra $B^\circ$ is a relative Frobenius lift. We say $\Phi$ is {\it invertible} if the restriction of $\Phi$ to $B^\circ$ is invertible (i.e., bijective) in which case $\Phi$ itself is, of course, bijective.

\

For each perfectoid ${\mathbb K}$-algebra $B$ one can define a {\it perfectoid space} $\Spa(B,B^\circ)$;
 cf. \cite[Def. 2.13, 6.1, 6.15]{scholze} or \cite[Def. 7.3.1, 9.1.1]{bhatt}. These are also called {\it affinoid perfectoid spaces} and form the local charts of global projective spaces. We also recall that
 the if $U(1/g)$ is the rational open set of a perfectoid space $\Spa(B,B^\circ)$, defined by  $g\in B^\circ$, then we have an isomorphism  of perfectoid spaces
\begin{equation}
\label{mamama}
U(1/g)\simeq \Spa(B^\circ[1/g])\h[1/p],(B^\circ[1/g])\h).\end{equation}
 Indeed, in view of \cite[Def. 2.13, Prop. 2.15]{scholze}, 
$$U(1/g)\simeq \Spa(B\langle 1/g\rangle,B\langle 1/g\rangle^+)$$
 where $B\langle 1/g\rangle^+$ is the $p$-adic completion of the integral closure 
of $B^\circ[1/g]$ in $B[1/g]$ and $B\langle 1/g\rangle$ is the completion of $B[1/g]$ with respect to the group topology of $B[1/g]$ defined by the subgroups $p^nB^\circ[1/g]$, $n\geq 1$. 
Clearly $B\langle 1/g\rangle=B^\circ[1/g]\h[1/p]$.
So
it is enough to check that $B^\circ[1/g]$ is integrally closed  in $B[1/g]$.
But $B^\circ$ is totally integrally closed in $B$ so $B^\circ$ is integrally closed in $B$
so $B^\circ[1/g]$ is integrally closed in $B[1/g]$ by the compatibility of integral closure with fractions, which proves Equation~\ref{mamama}.

\

A morphism of perfectoid spaces will be called here a {\it closed immersion} if locally on the target it is of the form $\Spa(B,B^\circ)\ra \Spa(C,C^\circ)$ where $B^\circ\ra C^\circ$ is a surjection of perfectoid ${\mathbb K}^\circ$-algebras.

\subsection{Principal covers}
Our ultimate goal is to construct perfectoid spaces given schemes. On the affine subschemes, we will first construct affinoid perfectoid spaces built from Frobenius lifts. To control the process of gluing these affinoid perfectoid spaces, we utilize principal covers of schemes. 

In what follows, we will say that a family of  open sets $(X_i)_{i\in I}$ in a scheme $X$ is a {\it principal cover} if $X=\cup X_i$, each $X_i$ is affine, $X_i=\text{Spec}\ B_i$, and for each $i,j\in I$ 
the intersection $X_i\cap X_j$ is a principal open set in both $X_i$ and $X_j$, i.e., there exist
$s_{ij}\in B_i$ and $s_{ji}\in B_j$ such that
\begin{equation}
\label{principalcover}
X_i\cap X_j=\text{Spec}\ (B_i)_{s_{ij}}=\text{Spec}\ (B_j)_{s_{ji}}.\end{equation}

\

A principal open cover $(X_i)_{i\in I}$ is called a {\it principal basis} if it forms a basis for the topology of $X$. 
Any affine scheme has a {\it canonical principal basis} consisting of all its principal open sets. It is immediate to see that any quasi-projective scheme $X$ over an affine scheme has a principal basis. Indeed, when $X$ is open in $\text{Proj}\ A$ then a principal basis for $X$ is the set of principal open sets of $\text{Proj}\ A$ contained in $X$. Note however that such a principal basis is not canonical. If $X\simeq \text{Proj}\ A$ and $X\simeq \text{Proj}\ B$ then the principal bases on $X$ induced from these isomorphisms
may not be compatible in the sense that their union is not necessarily a principal basis. For instance if $X$ is, say,  a projective curve and $X\simeq \text{Proj}\ A$ and $X\simeq \text{Proj}\ B$ are induced by two embeddings of $X$ into projective spaces attached to linear systems of the form $|mP|$ and $|mQ|$ where $P$ and $Q$ are points of $X$ whose difference is not torsion in $\text{Pic}(X)$ then the intersection of $U:=X\backslash P$ and $V:=X\backslash Q$ is not principal in either  $U$ or $V$.

\

Note the following construction of principal bases.  Assume $X$ is a scheme with a principal covering $(X_i)$ and consider a new covering  of $X$ consisting of all 
principal open sets of all the $X_i$'s; we refer to this new covering as the {\it canonical refinement} of $(X_i)$. 

\begin{lemma}
The canonical refinement of a scheme $X$ with principal cover $(X_i)$ is a principal basis. 
\end{lemma}
\begin{proof}
It is clearly a basis. To show it is principal let $Y_i$ be principal in $X_i$ and $Y_j$ be principal in $X_j$. We need to show that $Y_i\cap Y_j$ is principal in $Y_i$. 
Since $Y_i$ is principal in $X_i$ and $X_i\cap X_j$ is principal in $X_i$ it follows that $Y_i\cap X_j$ is principal in both
$Y_i$ and $X_i\cap X_j$. Similarly $Y_j\cap X_i$ is principal in both
$Y_j$ and $X_i\cap X_j$. Since both $Y_i\cap X_j$ and $X_i\cap Y_j$ are principal in $X_i\cap X_j$ it follows that their intersection $Y_i\cap Y_j$ is principal in $X_i\cap X_j$.
Since $Y_i\cap X_j$ is also principal in $X_i\cap X_j$ it follows  that $Y_i\cap Y_j$ is principal in $Y_i\cap X_j$. Indeed, it is a straightforward exercise to show that if $U\subset V\subset W$ are open immersions of affine schemes such that both $U$ and $V$ are principal in $W$ then $U$ is principal in $V$.  Since $Y_i\cap X_j$ is principal in $Y_i$ we get that $Y_i\cap Y_j$ is principal in $Y_i$. It is easy to check that if $U\subset V\subset W$ are open immersions of affine schemes such that  $U$ is principal in $V$  and $V$ is principal in $W$ then $U$ is principal in $W$. This proves our claim.
\end{proof}

We describe now a similar setup for perfectoid spaces. 

\begin{definition}
Assume  ${\mathbb X}$  is a perfectoid space. A family of  open sets $({\mathbb X}_i)_{i\in I}$ in ${\mathbb X}$ is a {\it principal cover} if ${\mathbb X}=\cup {\mathbb X}_i$, each  ${\mathbb X}_i=\Spa(B_i,B_i^\circ)$, with $B_i$ perfectoid, and for each $i,j\in I$ 
the intersection ${\mathbb X}_i\cap {\mathbb X}_j$ is a principal in both ${\mathbb X}_i$ and ${\mathbb X}_j$ in the sense that there exist
$s_{ij}\in B_i^\circ$ and $s_{ji}\in B_j^\circ$ such that
\begin{equation}
\label{principalcover2}
{\mathbb X}_i\cap {\mathbb X}_j=\Spa(((B_i^\circ)_{s_{ij}})^{\widehat{\ }}[1/p],
((B_i^\circ)_{s_{ij}})^{\widehat{\ }})=\Spa(((B_j^\circ)_{s_{ji}})^{\widehat{\ }}[1/p],
((B_j^\circ)_{s_{ji}})^{\widehat{\ }}).\end{equation}
\end{definition}

In general, perfectoid spaces do not a priori possess principal covers. Those that do may not a priori possess principal covers that form a basis for the topology.

\

Given a finite family of perfectoid algebras $(B_i)$, where each $B_i$ is an $R$-algebra, and elements $s_{ij}\in B_i$ and isomorphisms
 $\widehat{(B_i^\circ)_{s_{ij}}}\simeq \widehat{(B_j^\circ)_{s_{ji}}}$ satisfying the obvious cocycle condition one can attach to this data a perfectoid space ${\mathbb X}$ by gluing along the sets in Equation~\ref{mamama}, the perfectoid spaces
 $\Spa(B_i,B_i^\circ)$; the space ${\mathbb X}$ will come equipped with a natural principal cover.

Conversely, given a perfectoid space ${\mathbb X}$ with an affinoid cover ${\mathbb X}_i=\Spa(B_i,B_i^\circ)$, each $B_i$ an $R$-algebra, and equipped with a principal, cover one can construct a scheme
${\mathbb X}_k$ over $k$ by gluing, in the obvious way, the schemes
$$\Spec \ B_i^\circ/{\mathbb K}^{\circ \circ}B_i^\circ$$
via the identifications Equation~\ref{principalcover2}. We call 
${\mathbb X}_k$ the {\it reduction of ${\mathbb X}$ modulo ${\mathbb K}^{\circ \circ}$}.
 The construction ${\mathbb X}\mapsto {\mathbb X}_k$ depends on the principal cover
  but is functorial with respect to  morphisms in the category of perfectoid spaces that  respect the principal covers in the sense that they pull back members of a principal cover into open sets that are unions of members of principal covers.

\begin{definition}
Given a perfectoid space with a principal cover, by a {\it relative Frobenius lift} $\Phi$ on it we understand 
a collection of compatible relative Frobenius lifts $(\Phi_i)$ on the perfectoid algebras that define the principal cover. 
We say that $\Phi$ is {\it invertible} if all $\Phi_i$ are perfect i.e., ring isomorphisms.
\end{definition}

Finally, we collect this data categorically; verifying the definitions are straightforward. 

\begin{definition}
The category ${\mathcal P}_{\Phi}$ has objects are perfectoid spaces over ${\mathbb K}$ equipped with a principal cover and a relative Frobenius lift, and whose morphisms are the morphisms of perfectoid spaces that are compatible, as above, with the principal covers and with the relative Frobenius lifts. 
\end{definition}

We also obtain a natural functor, namely the {\it reduction modulo ${\mathbb K}^{\circ \circ}$} functor 
$$\text{res}:{\mathcal P}_{\Phi}\ra {\mathcal S}_k,\ \ \ {\mathbb X}\mapsto {\mathbb X}_k,$$  to the category of $k$-schemes; as described in the introduction.
 
\section{Perfectoid spaces attached to smooth schemes}

\subsection{The functor $P$}
Our first aim is to construct a functor 
$$(A,\phi)\mapsto P(A,\phi)$$
 from the category of 
 $p$-torsion free, $p$-adically complete, $R$-algebras for which $A/pA$ is reduced, equipped with a Frobenius lift $\phi$,
 to the category of perfectoid ${\mathbb K}$-algebras equipped with a relative Frobenius lift. We need a series of Lemmas. The first has a straightforward proof which we leave to the reader.

\begin{lemma}\label{l1} \ 

1) 
Let $A\ra B$ be a ring homomorphism such that  $A$ is $p$-adically separated and $B$ is $p$-torsion free. Assume the induced map $A/pA\ra B/pB$ is injective. Then $A\ra B$ is injective. 

2) Let $A\ra B$ be a ring homomorphism such that  $A$ is $p$-adically complete and $B$ is $p$-adically separated. Assume the induced map $A/pA\ra B/pB$ is surjective. Then $A\ra B$ is surjective. 

\end{lemma}

\begin{lemma}\label{l2}
Let $A_1\ra A_2\ra A_3\ra...$ be a sequence of ring homomorphisms such that
the homomorphisms $A_1/pA_1\ra A_2/pA_2\ra A_3/pA_3\ra ...$ are injective and $A_n$ are
$p$-torsion free,  $p$-adically separated, and totally integrally closed in $A_n[1/p]$ for all $n$.
The limit $$A=\varinjlim A_n$$ is $p$-torsion free and  totally integrally closed in $A[1/p]$.
\end{lemma}

\begin{proof}
By Lemma \ref{l1} we may assume the maps $A_1\ra A_2\ra A_3\ra...$ are injections and up to conflating $A_i$ with its image, we may assume these are inclusions. 
By our injectivity hypothesis we have 
$$pA_{n+1}\cap A_n=pA_n$$
hence, since  $A_{n+1}$ is $p$-torsion free we get, by induction, that
$$p^kA_{n+1}\cap A_n=p^kA_n$$
for all $k$. This implies
$$A_{n+1}\cap \frac{1}{p^k} A_n=A_n,$$
hence
$A_{n+1}\cap A_n[1/p]=A_n$,
thus
$$A\cap A_n[1/p]=A_n.$$
Clearly $A$ is $p$-torsion free. Assume now $\alpha\in A[1/p]$ and $\alpha^{\mathbb N}$
is contained in a finitely generated $A$-submodule of $A[1/p]$, hence 
$$\alpha^{\mathbb N}\subset \frac{1}{p^{N_0}}A$$
 for some $N_0$. On the other hand $\alpha\in A_{n_0}[1/p]$ for some $n_0$. Hence
 $$p^{N_0} \alpha^{\mathbb N}\subset  A \cap A_{n_0}[1/p]=A_{n_0},\text{ so } \alpha^{\mathbb N}\subset \frac{1}{p^{N_0}}A_{n_0}$$
 Since $A_{n_0}$ is totally integrally closed in $A_{n_0}[1/p]$ it follows that $\alpha\in A_{n_0}$ so $\alpha\in A$ and we are done.
 \end{proof}

\newpage 

We now turn our attention to rings containing compatible systems of $p$-power roots of $p$. Recall the following result on total integrally closed rings and localization \cite[Lem. 5.1.3]{bhatt}. 

\begin{lemma}
\label{l3}
Assume $A$ is a $p$-torsion free ring admitting a compatible system of $p$-power roots
$p^{1/p^n}$. Assume $A$ is totally integrally closed in $A[1/p]$. Then $\widehat{A}$ is totally integrally closed in $\widehat{A}[1/p]$.
\end{lemma}

In conjunction, Lemmas~\ref{l2} and \ref{l3} combine to essentially show total integral closure persists along direct limits of endomorphisms identity along such $p$-power root systems.

\begin{lemma}
\label{l4}
Let $A$ be a $p$-torsion free $p$-adically separated ring containing a compatible system of $p$-power roots
$\{p^{1/p^n}\}$. 
Assume $A$ is equipped with a ring endomorphism $\Phi:A\ra A$ that is the identity on $\{p^{1/p^n}\}$ and assume the induced homomorphism $\overline{\Phi}:A/pA\ra A/pA$ is injective. 
 Assume moreover that $A$ is totally integrally closed in $A[1/p]$. Let
 $$C:=\varinjlim_{\Phi} A.$$
The completion $\widehat{C}$ is $p$-torsion free and totally integrally closed in $\widehat{C}[1/p]$.
\end{lemma}

\begin{proof}By Lemma \ref{l2}, $C$ is $p$-torsion free and totally integrally closed in $C[1/p]$.
In particular  $\widehat{C}$ is also $p$-torsion free.
On the other hand since $\Phi$ is the identity on $\{p^{1/p^n}\}$ the latter defines
a compatible system of $p$-power roots of $p$ in $C$. So, by Lemma \ref{l3},  $\widehat{C}$ is totally integrally closed in $\widehat{C}[1/p]$.
\end{proof}

We next address how these properties change under adjunction of $p$-power compatible root systems and tensoring by ${\mathbb K}^\circ$. 

\begin{lemma}\label{l5}
Let $A$ be a $R$-algebra that is $p$-torsion free, $p$-adically separated,  and for which $A/pA$ is reduced.
Let $A_m:=A\otimes_R R[p^{1/p^m}]$. Each $A_m$ is $p$-torsion free, $p$-adically separated,  and totally integrally closed in $A_m[1/p]$.
\end{lemma}

\begin{proof}
Each $A_m$ is $p$-torsion free and $p$-adically separated because it is a direct sum of finitely many copies of $A$. Assume now $\alpha \in A_m[1/p]$ satisfies 
$$\alpha^{\mathbb N}\subset
\frac{1}{p^N}A_m.$$
 Write $\pi_m=p^{1/p^m}$ and $\alpha=\frac{a}{\pi_m^{\nu}}$ with $a\in A_m\backslash \pi_mA_m$ and $\nu\geq 0$. Assume $\nu\geq 1$ and seek a contradiction. We get that for any $n\geq 1$,
 $$a^n\in \pi_m^{n\nu -p^mN}A_m.$$
 Choose $n$ such that $n\nu> p^mN$. Then the image of $a^n$ in $A_m/\pi_mA_m$ is $0$. But $A_m/\pi_mA_m\simeq A/pA$ is reduced. So $a\in \pi_mA_m$, a contradiction.
\end{proof}

\begin{lemma}\label{l6}
Assume $A$ is as in Lemma \ref{l5} and furthermore set ${\mathcal A}=A\otimes_R {\mathbb K}^\circ$.
\begin{enumerate}

\item $\widehat{\mathcal A}$ is $p$-torsion free and totally integrally closed in $\widehat{\mathcal A}[1/p]$.

\item We have a natural isomorphism
\begin{equation}
\label{astept}
\widehat{\mathcal A}/{\mathbb K}^{\circ \circ}\widehat{\mathcal A}\simeq A/pA.
\end{equation}
\end{enumerate}
\end{lemma}
\begin{proof}

By Lemma \ref{l2}, in order to check assertion 1  it is enough to check that ${\mathcal A}$ is $p$-torsion free and totally integrally closed in ${\mathcal A}[1/p]$.
With $A_m$ as in Lemma \ref{l5} we have
${\mathcal A}=\varinjlim A_m.$
In view of Lemma \ref{l2} and \ref{l5}, it is enough to check that $A_m/pA_m\ra A_{m+1}/pA_{m+1}$ are injective. These maps are injective because they are obtained  from 
the reduction modulo $p$ of the injective map $R[p^{1/p^m}]\ra R[p^{1/p^{m+1}}]$ via tensoring with $A/pA$ over the field $R/pR$. This ends the proof of assertion 1. Assertion 2 
follows from a direct computation.
\end{proof}

We now incorporate Frobenius lifts to the picture. 

\begin{lemma}\label{l7}
Assume $A$ as in Lemma \ref{l5} and ${\mathcal A}=A\otimes_R {\mathbb K}^\circ$.
Assume $A$ is equipped with a Frobenius lift $\phi$ 
and consider the unique endomorphism
$\Phi$ on ${\mathcal A}$ compatible with $\phi$ and with the relative Frobenius lift on ${\mathbb K}^\circ$. Also consider the extension of $\Phi$, abusively denoted by $\Phi$, on $\widehat{\mathcal A}$.
Set $${\mathcal B}=\varinjlim_{\Phi} \widehat{\mathcal A}.$$
\begin{enumerate}

\item The completion $\widehat{\mathcal B}$ is $p$-torsion free and totally integrally closed in $\widehat{\mathcal B}[1/p]$.

\item Frobenius is surjective on ${\mathcal B}/p{\mathcal B}$.

\item The endomorphism $\Phi$ of $\widehat{\mathcal A}$  is a relative Frobenius lift and so is the induced endomorphism $\Phi$ of $\widehat{\mathcal B}$.

\item There is a canonical isomorphism
$$\widehat{\mathcal B}/{\mathbb K}^{\circ \circ}\widehat{\mathcal B}\simeq (A/pA)_{\text{perf}}$$
and the endomorphism induced by $\Phi$ on $\widehat{\mathcal B}/{\mathbb K}^{\circ \circ}\widehat{\mathcal B}$ corresponds to the Frobenius on $(A/pA)_{\text{perf}}$.
\end{enumerate}
\end{lemma}
\begin{proof} Let us  prove assertion 1. In view of Lemmas \ref{l4} and \ref{l6}, it is enough to check that $\overline{\Phi}$
is injective on 
$${\mathcal A}/p{\mathcal A}=(A/pA)\otimes _{R/pR}({\mathbb K}^\circ/p{\mathbb K}^\circ).$$
It is enough to check that  $\overline{\Phi}$
is injective on 
$$A/pA\otimes _{R/pR}(R[p^{1/p^m}]/(p))=(A/pA)[z_m]/(z_m^{p^m}).$$
which is clear because $\overline{\Phi}$ acts like Frobenius on $A/pA$, which is reduced, and fixes the class of  $z_m$.

To prove assertion 2, note that 
$${\mathcal B}/p{\mathcal B}=\lim_{\stackrel{\ra}{\Phi}} {\mathcal A}/p{\mathcal A}.$$
To show that any element of ${\mathcal B}/p{\mathcal B}$ is a $p$-th power it is enough to show that for any element 
$f\in {\mathcal A}/p{\mathcal A}$ there is an element $g\in {\mathcal A}/p{\mathcal A}$ such that
$$(f,n)\sim (g^p,n+1)\in {\mathcal A}/p{\mathcal A}\times {\mathbb N}$$ i.e.,  $\Phi(f)=g^p$. It is enough to check this 
for $f\in \overline{A}$ and for $f\in {\mathbb K}^\circ/p{\mathbb K}^\circ$ separately. This follows from the fact that
$\overline{\Phi}$ acts as Frobenius on $\overline{A}$ and as the relative Frobenius
on ${\mathbb K}^\circ/p{\mathbb K}^\circ$, so in the first case we can take $g=f$ while in the second case
we can take $g$ to be any $p$-root of $\Phi(f)$, which is possible because Frobenius is surjective
on ${\mathbb K}^\circ/p{\mathbb K}^\circ$.

Finally  assertion 3 follows from the isomorphism Equation~\ref{astept} from Lemma~\ref{l6} and assertion 4 also follows.
\end{proof}

Our discussion leads to the first theorem which is an affine form of the theorems discussed in the introduction. 

\begin{theorem}\label{r1}
Let $A$ be a $R$-algebra that is $p$-torsion free, $p$-adically complete,  and for which $A/pA$ is reduced. Let ${\mathcal A}=A\otimes_R {\mathbb K}^\circ$.
Assume $A$ is equipped with a Frobenius lift $\phi$, consider the unique endomorphism 
$\Phi$ on ${\mathcal A}$ compatible with $\phi$ and with the relative Frobenius lift on ${\mathbb K}^\circ$, and consider the  extension $\Phi$ on $\widehat{\mathcal A}$.
Let $${\mathcal B}=\varinjlim_{\Phi} \widehat{\mathcal A},\ \ \ 
{\mathcal B}':=\varinjlim_{\Phi} {\mathcal A},\ \ \ 
P^\circ(A,\phi):=\widehat{\mathcal B}=\widehat{{\mathcal B}'},\ \ \ 
P(A,\phi):=\widehat{\mathcal B}[1/p]=\widehat{{\mathcal B}'}[1/p].
$$
The following hold:
\begin{enumerate}
\item The ${\mathbb K}$-algebra $P(A,\phi)$, with its natural Banach structure, is perfectoid.

\item We have
$$P(A,\phi)^\circ=P^\circ(A,\phi).$$

\item The endomorphism of $P(A,\phi)$ induced by $\Phi$ (which we still denote by $\Phi$) is an invertible relative Frobenius lift.

\item We have an isomorphism
 $$\widehat{\mathcal B}/{\mathbb K}^{\circ \circ}\widehat{\mathcal B}\simeq (A/pA)_{\text{perf}}.$$

 \item The rule $(A,\phi)\mapsto (P(A,\phi),\Phi)$ defines a functor
from the category of $p$-torsion free, $p$-adically complete $R$-algebras  for which $A/pA$ is reduced, equipped with a Frobenius lift,  to the category of perfectoid ${\mathbb K}$-algebras equipped with a  relative Frobenius lift.

\item The functor in 5) sends surjections into surjections.
\end{enumerate}
\end{theorem}
\begin{proof}
Assertion 6 in the Theorem follows from the fact that the functors 
$$\varinjlim,\ \otimes_R{\mathbb K}^\circ,\ \ \widehat{\ }$$
send surjections into surjections; cf. \cite[Thm. 8.1 and Appendix A]{matsumura}.
The only assertion that remains to be checked in that $\Phi$ 
in assertion 3 is invertible, i.e., that it 
is bijective on $P^\circ(A,\phi)$.
This can be shown as follows. Since $A/pA$ is reduced the reduction modulo $p$ of $\phi$ is an injective map $A/pA\ra A/pA$. By assertion 1 in Lemma \ref{l1},
 $\phi$ is injective on $A$. Since ${\mathbb K}^\circ$ is $R$-flat it follows that $\Phi$ is injective on ${\mathcal A}$ hence $\Phi$ is injective on ${\mathcal B}'$. But clearly 
$\Phi$ is surjective, and hence bijective, on ${\mathcal B}'$. So $\Phi$ is bijective on $P^\circ(A,\phi)$.
\end{proof}

\begin{remark}
The  functor $P$ may be extended to the category of not necessarily $p$-adically complete
$p$-torsion free $R$-algebras  for which $A/pA$ is reduced, equipped with a Frobenius lift, via the formula $P(A,\phi)=P(\widehat{A},\phi)$.
\end{remark}

The next lemma will be fundamental to globalizing Theorem~\ref{r1}. 

\begin{lemma}
\label{d1}
For any $s\in A$ we claim that have a canonical isomorphism
$$P^\circ((A_s)^{\widehat{\ }},\phi)\simeq ((P^\circ(A,\phi))_s)^{\widehat{\ }},$$
where $\phi$ in the left hand side is a naturally induced Frobenius lift on $(A_s)^{\widehat{\ }}$ and the subscript $s$ in the right hand side signifies taking fractions with respect to the image of $(s,1)$. 
\end{lemma}
\begin{proof}
This follows because, when
$\phi(s)\in s^p+pA$, we  have induced maps
$$\phi:(A_s)^{\widehat{\ }}\ra (A_{\phi(s)})^{\widehat{\ }}=(A_s)^{\widehat{\ }},$$
and
$$\Phi:((A_s)^{\widehat{\ }}\otimes_R {\mathbb K}^\circ)^{\widehat{\ }}\ra ((A_{\phi(s)})^{\widehat{\ }}\otimes_R {\mathbb K}^\circ)^{\widehat{\ }}=((A_s)^{\widehat{\ }}\otimes_R {\mathbb K}^\circ)^{\widehat{\ }}.$$
On the other hand
$$((A_s)^{\widehat{\ }}\otimes_R {\mathbb K}^\circ)^{\widehat{\ }}\simeq ({\mathcal A}_s)^{\widehat{\ }},\ \ {\mathcal A}=A\otimes_R {\mathbb K}^\circ$$ and
$$\lim_{\stackrel{\ra}{\Phi}} (({\mathcal A}_s)^{\widehat{\ }})\simeq ((\lim_{\stackrel{\ra}{\Phi}} \widehat{\mathcal A})_s)^{\widehat{\ }},$$
which ends the proof of our claim.
\end{proof}

The construction in Theorem \ref{r1} possesses an expected universality property.

\begin{theorem}\label{uniuni} Assume the notation of Theorem \ref{r1} and let 
$\iota:A\ra P^\circ(A,\phi)$ be the homomorphism defined by letting
$\iota(a):=[a\otimes 1,1]\in {\mathcal B}'$, $a\otimes 1\in A\otimes_R{\mathbb K}^\circ$.
Also let $C$ be a $p$-adically complete ${\mathbb K}^\circ$-algebra
equipped with an invertible relative Frobenius lift $\Phi_C$. 

For any $R$-algebra homomorphism
$u_1:A\ra C$ satisfying $$u_1\circ \phi=\Phi_C \circ u_1$$ there exists a unique ${\mathbb K}^\circ$-algebra homomorphism $u:P^\circ(A,\phi)\ra C$ such that 
$$u\circ \iota=u_1,\ \ \ \text{and}\ \ \ \Phi_C\circ u=u_1\circ \Phi.$$
\end{theorem}

\begin{proof} The existence of $u$ is an easy exercise using the fact, since $\Phi_C$ is bijective, we have 
$C=\varinjlim_{\Phi_C} C.$
To prove the uniqueness of $u$ 
assume 
$$u,v:P^\circ(A,\phi)\ra C$$
are two ${\mathbb K}^\circ$-algebra homomorphisms that commute with the actions of $\Phi$ and $\Phi_C$ and such that $u\circ \iota=v\circ \iota$. 
  It is enough to show that $u$ and $v$ coincide on ${\mathcal B}'$.

Set $\alpha\in A\otimes_R{\mathbb K}^\circ.$ In order to show that $u([\alpha,i])=v([\alpha,i])$
we may assume $\alpha=a\otimes 1$ with $a\in A$. Since $\Phi_C$ is injective it is enough to show that $$\Phi_C^i(u([a\otimes 1,i]))=\Phi_C^i(v([a \otimes 1,i])).$$ But
$$\Phi_C^i(u([a\otimes 1,i]))=u(\Phi^i([a \otimes 1,i]))=u([a\otimes 1,1])=u(\iota(a)),$$
and a similar equality  holds for $v$ which gives the claim. 
\end{proof}

\subsection{The functor $J^{\infty}$}

We recall now the fundamental construction concerning $p$-jets, for more see \cite{char,pjets,equations}. These are certain functors $B\mapsto J^n(B)$ from the category of finitely generated $R$-algebras $B$ to the category of $p$-adically complete $R$-algebras as follows. For rings of finite type over $R$, these are explicit. When $B = R[x]$ where $x$ is a finite tuple of variables, we obtain the free $\delta$-ring $R[x,x',\ldots]$ on the set of variables $x$ which has its $p$-derivation given by 
$$\d:R[x,x',...]\ra R[x,x',...],\ \ \d(F)=\frac{\phi(F)-F^p}{p}$$
where
$$\phi:R[x,x',...]\ra R[x,x',...]$$
is the unique extension of the Frobenius lift on $R$ that sends
$$x\mapsto x^p+px',\ \ x'\mapsto (x')^p+px'',...,$$
we refer to this $\d$ as the {\it universal $p$-derivation}.

Now when $B=R[x]/(f)$, $x$ and $f$ finite tuples, we set
$$J^n(B):=(R[x,x',...,x^{(n)}]/(f,\d f,...\d^n f))^{\widehat{\ }}.$$ Note even when 
$B$ is  of finite type flat over $R$ the rings $J^n(B)$ are not necessarily $p$-torsion free, e.g. take $B=R[x]/(x^p)$.

It is a critical calculation, but easy to see for $s\in B$ there are functorial isomorphism $$J^n(B_s)^{\widehat{\ }} \simeq ((J^n(B))_s)^{\widehat{\ }}.$$ This means that one can glue to construct functors $$X\mapsto J^n(X)$$ from the category of schemes $X$ of finite type over $R$ to the category of $p$-adic formal schemes as follows. For any such scheme $X$  we take an affine open cover $$X=\bigcup_i \Spec  B_i$$
and we define $$J^n(X):=\bigcup_i \Spf J^n(B_i)$$
where $\Spf J^n(B_i)$ are glued naturally, using the appropriate universality property of this construction; cf. \cite[Sec. 2.1, pg. 358]{pjets} or \cite[Cor. 3.4]{equations} for more details. We refer to these as the {\it $p$-jet space of $X$}. 

\

\begin{remark} \label{d2} We record a few important facts about the functors $J^n$ which are easy to verify. 
\begin{enumerate}
\item The functors $J^n$ send closed immersions into closed immersions. 
\item By functoriality, if $G$ is a group scheme then $J^n(G)$ are naturally groups in the category of $p$-adic formal schemes. 
\item If $X$ is affine and possesses an \'{e}tale map $$X\ra {\mathbb A}^d=\Spec R[T]$$ where $T$ is a $d$-tuple of variables then, viewing $T$ as a tuple of elements of $\cO(X)$, we have a natural identification $$\cO(J^n(X))= \cO(\widehat{X})[T',...,T^{(n)}]^{\widehat{\ }},$$ where $T',...,T^{(n)}$ are $d$-tuples of indeterminates and $\d T=T'$, $\d T'=T''$, etc. In particular, if $X$ is any smooth scheme of finite type over $R$ of relative dimension $d$ then $J^{n+1}(X)\ra J^n(X)$ are locally trivial fibrations, in the Zariski topology, with fiber $\widehat{{\mathbb A}^d}$. In particular, the reduction modulo $p$ of $J^n(X)$ is reduced and hence $\phi$ is injective on $J^{\infty}(B)$ for $X=\text{Spec}\ B$. See \cite[Prop. 3.17]{equations} for more details. 
\end{enumerate}
\end{remark}

Set $J^\infty(B) = \varinjlim_n J^n(B)$. This has a naturally defined $p$-derivation $\d^{J^{\infty}(B)}$ coming from the universal $p$-derivation. Indeed $(J^{\infty}(B),\d^{J^{\infty}(B)})$ has the following universality property: for any pair $(C,\d^C)$ consisting of a ring $C$ and a $p$-derivation $\d^C$ on $C$ and for any ring homomorphism $u:B\ra C$ there is a unique 
extension $\tilde{u}:J^{\infty}(B)\ra C$ of $u$ such that $\d^C\circ \tilde{u}=\tilde{u}\circ \d^{J^{\infty}(B)}$.

When $B$ is smooth finitely generated over $R$, then by
 the structure of $\widehat{{\mathbb A}^d}$-fibration mentioned above one has natural inclusions
$$\widehat{B}=B_0\subset B_1\subset B_2\subset ...$$ 
whose reductions modulo $p$ are injective. Also each $B_n$ is $p$-adically separated and $p$-torsion free and $B_n/pB_n$ are smooth over $R/pR$, in particular they are reduced. 

In this case, $J^{\infty}(B)$ is then $p$-torsion free, $p$-adically complete,
and 
$$J^{\infty}(B)/pJ^{\infty}(B)=\varinjlim B_n/pB_n$$
is a reduced ring. If $X$ is a smooth, not necessarily affine, scheme over $R$ and if $X=\bigcup \Spec B_i$
is an affine open cover then, by \cite[Thm. 2.10]{pjets}, the proalgebraic $k$-scheme
$$\bigcup \Spec (J^{\infty}(B_i)/pJ^{\infty}(B_i))$$
is naturally isomorphic to the {\it Greenberg transform} $\text{Green}(X)$ of $X$, in the sense of \cite{greenberg}. To summarize we have the following. 

\begin{theorem}\label{r2}\cite{pjets}
The rule $B\ra (J^{\infty}(B),\phi)$ defines a functor from the category of smooth finitely generated $R$-algebras $B$ to the category of $p$-torsion free, $p$-adically complete, $R$-algebras, with reduced reduction modulo $p$, and equipped with a Frobenius lift. 
Moreover if $X=\bigcup \Spec \ B_i$ then $\bigcup \Spec J^{\infty}(B_i)/pJ^{\infty}(B_i)$ is isomorphic to the Greenberg transform $\text{Green}(X)$ of $X$.\end{theorem}

We end our discussion by using these $p$-jet spaces to making the following definition, which was mentioned in the introduction. 

\begin{definition}
Let $X$ and $Y$ be two smooth schemes over $R$. A {\it $\d$-morphism} of order $n$ from $X$ to $Y$ is an element of 
$$\Hom(J^n(X),J^0(Y))$$
where $\Hom$ is taken in the category of formal schemes over $R$.
A  {\it $\d$-morphism}  from $X$ to $Y$ is an element of 
$$\Hom_{\d}(X,Y):=\varinjlim_{n} \Hom(J^n(X),J^0(Y)).$$
\end{definition}

Given a $\d$-morphism of order $n$, 
$$f:J^n(X)\ra J^0(Y),$$ and a $\d$-morphism of order $m$, 
$$g:J^m(Y)\ra J^0(Z),$$ one can define a $\d$-morphism of order $m+n$, 
$$J^{n+m}(X)\ra J^0(Z)$$
by composing $g$ with the morphism 
$$J^m(f):J^{n+m}(X)\ra J^m(Y)$$ induced by $f$ via the universality property of $J^{\infty}$. One obtains a composition law for $\d$-morphisms that is trivially seen to be associative.

\subsection{The functor $P^{\infty}$}

Putting together Theorems \ref{r1},  \ref{r2} and  Remarks \ref{d1} and \ref{d2} we get the following theorem. 

\begin{theorem}
\label{ping}

\begin{enumerate}
\item The rule
$$B\mapsto P^{\infty}(B):=(P(J^{\infty}(B),\phi),\Phi)$$
defines a functor from the category smooth finitely generated $R$-algebras $B$
 to the category of perfectoid ${\mathbb K}$-algebras equipped with a relative Frobenius lift. 
 
\item For $P^\circ:=P^\circ(J^{\infty}(B),\phi)$ we have
 $$P^\circ/{\mathbb K}^{\circ \circ}P^\circ\simeq (J^{\infty}(B)/pJ^{\infty}(B))_{\text{perf}}$$.
 
\item For all $s\in B$, 
 $$P^\circ(J^{\infty}(B_s),\phi)\simeq ( 
(
P^\circ(J^{\infty}(B),\phi)
)_s
)^{\widehat{\ }}.$$

\item The functor $P^{\infty}$ sends surjections into surjections.

\item The relative Frobenius $\Phi$ is bijective on $P^\circ(J^{\infty}(B),\phi)$.
\end{enumerate}
\end{theorem}

\begin{remark}
In some ways, Theorem~\ref{ping} is related to \cite[Thm. 3.10]{BS18}. We will not utilize any prismatic language, but roughly the construction of $A/I$ in loc. cit as a perfectoid ring is very close to the proof we've provided of Theorem~\ref{ping}. However, our approach has utilized more extensively the language of $p$-jets which helps illuminate the connections to the applications mentioned in the introduction. It could be interesting to translate the other results we utilize into the prismatic setting and explore the ramifications, but these are beyond the scope of this paper, so we proceed without any further remark on these connections. 
\end{remark}

Assume now $X$ is a smooth quasi-projective scheme over $R$ and let $(X_i)_{i\in I}$ be a principal basis for $X$, $B_i=\cO(X_i)$, and $s_{ij}$ as in Equation~\ref{principalcover}.
In view of assertion 3 in the above Theorem~\ref{ping} plus Equation~\ref{mamama}, the perfectoid spaces $$\Spa(P(J^{\infty}(B_i),\phi), P^\circ(J^{\infty}(B_i),\phi))$$
glue together to yield a perfectoid space $P^{\infty}(X)$ equipped with a principal cover. Note however, it is not generally a basis for the topology of that space. 
If $X$ is affine, $X=\Spec B$, and $X_i$ are principal in $X$ then 
$$P^{\infty}(X)\simeq\Spa(P(J^{\infty}(B),\phi), P^\circ(J^{\infty}(B),\phi)).$$
Going back to an arbitrary smooth quasi-projective $X$, if $Y$ is another smooth quasi-projective scheme over $R$ and $(Y_j)_{j\in J}$ is  a principal basis for $Y$ then for any morphism of schemes $X\ra Y$, which we do not assume to be compatible with the principal 
bases, there is an induced morphism $P^{\infty}(X)\ra P^{\infty}(Y)$.
We obtain the following consequence.

\begin{corollary}
\label{pong}
There is  a  functor
$$P^{\infty}:{\mathcal S}\ra {\mathcal P}_{\Phi},\ \ \ X\mapsto P^{\infty}(X)$$
 from the category ${\mathcal S}$ of  quasi-projective smooth schemes over $R$ to the category ${\mathcal P}_{\Phi}$ of perfectoid spaces over ${\mathbb K}$ equipped with a principal cover and with a relative Frobenius lift having the following properties.

 \begin{enumerate}
 
\item The functor $P^{\infty}$ sends closed immersions into closed immersions.

\item For any $X$ in ${\mathcal S}$ the scheme 
 $P^{\infty}(X)_k$, i.e., reduction 
 modulo ${\mathbb K}^{\circ \circ}$ of $P^{\infty}(X)$ with respect to its principal cover, is isomorphic to the perfection $G(X)_{\text{perf}}$ of the Greenberg transform $\text{Green}(X)$ of $X$. 
 
 \item For any $X$ in ${\mathcal S}$ the relative Frobenius lift $\Phi$ on $P^{\infty}(X)$ is invertible.

\end{enumerate}

\end{corollary}

\begin{remark}
The functor $P^{\infty}$ in mentioned in introduction. The construction required us to arbitrarily attach to each quasi-projective smooth scheme $X$ some principal basis
$(X_i)_{i\in I}$ and define $P^{\infty}$ using this choice of principal bases. Thus $P^{\infty}$ depends on the choice of the principal bases. However,  if $\tilde{P}^{\infty}$ is the functor corresponding to a different choice of principal bases then $P^{\infty}$ and $\tilde{P}^{\infty}$ are canonically isomorphic as functors.
\end{remark}

\begin{remark}\label{extensionn}
Let ${\mathcal S}_{\d}$ be the category whose objects are the smooth quasi-projective schemes over $R$ and whose morphisms are the $\d$-morphisms. We claim that the functor
$P^{\infty}$ in Corollary \ref{pong} naturally extends to a functor, still denoted by, 
$$P^{\infty}:{\mathcal S}_{\d}\ra {\mathcal P}_{\Phi}.$$
This plus our above discussion proves Theorem \ref{soare} in the introduction.
To check our claim,
 let $X$ be a smooth scheme with a principal basis $(X_i)$, $X_i=\text{Spec}\ B_i$. 
In view of 
Remark \ref{d1} and Equation \ref{mamama}
the perfectoid spaces 
$$P_{X_i,f}:=\Spa(P((J^{\infty}(B_i)_f)^{\widehat{\ }},\phi), P^\circ((J^{\infty}(B_i)_f)^{\widehat{\ }},\phi)),$$
where for each $i$, $f$ runs through the elements of the ring $J^{\infty}(B_i)$,
glue together to yield a perfectoid space which we temporarily denote $P^{\infty}_*(X)$. 
It is easy to see that there is a canonical isomorphism $P^{\infty}(X)\simeq P^{\infty}_*(X)$ for each object $X$ so the interest here is to show 
how $\delta$-morphisms induce maps via functoriality. This is easier done using the $P^{\infty}_*(X)$ description. Indeed for any $\d$-morphism $J^n(X)\ra J^0(Y)$  one has induced 
morphisms
$$\cO(J^0(Y_j))_g\ra \cO(J^n(X_i))_f$$
for various appropriate $f,g,i,j$ and hence, by the universality property of $p$-jet spaces, morphisms
$$\varinjlim{_m} \cO(J^m(Y_j))_g\ra \varinjlim_m
\cO(J^{n+m}(X_i))_f=\varinjlim_m \cO(J^m(X_i))_f$$
inducing morphisms
$$P_{X_i,f}\ra P_{Y_j,g}.$$
The latter which glue together to give a morphism $P^{\infty}_*(X)\ra P^{\infty}_*(Y)$, hence a morphism $P^{\infty}(X)\ra P^{\infty}(Y)$.

Note that the covering
$$(\Spec J^{\infty}(B_i)/pJ^{\infty}(B_i))_i$$
of the Greenberg transform of $X$ is a principal covering but, of course, not a principal basis. 
However, the canonical refinement of the above covering,
$$(\Spec J^{\infty}(B_i)_f/pJ^{\infty}(B_i)_f)_{i,f},$$
is a principal basis.  
\end{remark}

\section{The case of  ${\mathbb G}_m$}

 We illustrate the theory on the example $X = {\mathbb G}_m$. Of particular note is that by functoriality the multiplication by $p$ map on $\mathbb{G}_m$, which we denote $[p]$ induces a map $P^{\infty}([p])$ on $P^{\infty}({\mathbb G}_m)$ which we explore in detail. Let 
  $${\mathbb G}_m=\Spec R[x,x^{-1}],$$
   with $x$ a variable,  be the multiplicative group scheme over $R$, and let
 $$[p]:{\mathbb G}_m\ra {\mathbb G}_m$$ be the isogeny induced by the $R$-algebra homomorphism, still denoted by
 $$[p]:R[x,x^{-1}]\ra R[x,x^{-1}],$$
  defined by $[p](x)=x^p$. We next analyze the induced morphism 
  $$P^{\infty}([p]):P^{\infty}({\mathbb G}_m)\ra P^{\infty}({\mathbb G}_m),$$
   equivalently the induced morphism of Banach ${\mathbb K}$-algebras 
 \begin{equation}
 \label{barbecue}
 P^{\infty}([p]):P^{\infty}(R[x,x^{-1}])\ra P^{\infty}(R[x,x^{-1}]).\end{equation}

\begin{thm}\label{poi}
The morphism \ref{barbecue}
 is surjective but not injective.
\end{thm}

In order to prove our theorem we need some notation. Set $B=R[x,x^{-1}]$, and 
$$B_n:=\cO(J^n(B))=R[x,x^{-1},x',...,x^{(n)}]\widehat{\ }.$$
Hence for
$${\mathcal A}:=J^{\infty}(B)\otimes_R {\mathbb K}^{\circ \circ}=R[x,x^{-1},x',x'',...]^{\widehat{\ }}\otimes_R {\mathbb K}^{\circ \circ},$$
we have $\widehat{\mathcal A}={\mathbb K}^{\circ \circ}[x,x^{-1},x',x'',...]^{\widehat{\ }}$
and $$P^{\infty}(B)^\circ=\widehat{\mathcal B},\ \ \ {\mathcal B}:=\varinjlim_{\Phi} \widehat{\mathcal A},\ \ \ P^{\infty}(B)=P^{\infty}(B)^\circ[1/p].$$
We need to record  the following consequence of \cite[Thm. 1.1]{pfin1}.
Set 
$$y_n:=\d^n(x^p)\in B_n,\ \ \ n\geq 0.$$ For us, the consequence of \cite[Thm 1.1]{pfin1} that we need is the following lemma.

\begin{lemma} There exist elements $v_n, w_n\in B_n$, $n\geq 1$ such that:

1) $y_1=pw_1$.

2) $y_2=x^{p^2(p-1)}(x')^p+pw_2$

2) For $n\geq 3$,  $y_n=x^{p^n(p-1)}(x^{(n-1)})^p+v_{n-2}+pw_n$. 
\end{lemma}

\

{\it Proof of Theorem \ref{poi}}. We start by proving surjectivity.

Let $C$ be the image of $P^{\infty}([p]):\widehat{\mathcal B}\ra \widehat{\mathcal B}$. 
Theorem \ref{poi} will be proved if we prove that $C=\widehat{\mathcal B}$. 
By Lemma \ref{l1} it is enough to show that any element in $\widehat{\mathcal B}$ is congruent modulo $p$ to an element of $C$. Recall that the elements of $\widehat{\mathcal B}$ are represented by pairs
$(a,i)\in \widehat{\mathcal A}\times {\mathbb N}$. So it is enough to check that for any $n\geq 0$ and $i\geq 1$ the class 
$$[x^{(n)},i]\in \widehat{\mathcal B}$$
is congruent modulo $p$ in $\widehat{\mathcal B}$ to an element of $C$. We proceed by induction on $n$. Denote by $\equiv$ congruence modulo $p$ in $\widehat{\mathcal B}$.
 Now for any $i\geq 1$ we have
 $$[x,i]= [\phi(x),i+1]\equiv [x^p,i+1]=P^{\infty}([p])[x,i+1]\in C$$ 
 so the case $n=0$ is checked.
 For the induction step assume our assertion is true for all $k$ with  $1\leq k\leq n-1$ for some $n\geq 1$. Set $v_0=0$. By the induction hypothesis 
 $$[x^{-p^{n+1}(p-1)}v_{n-1},i+1]\in C+p\widehat{\mathcal B}.$$ Hence
 $$\begin{array}{rcl}
 [x^{(n)},i] & = &  [\phi(x^{(n)}),i+1]\\
 \ & \ & \ \\
 \ & \equiv  & [(x^{(n)})^p,i+1]\\
 \ & \ & \ \\
 \ & \equiv  & [x^{-p^{n+1}(p-1)}y_{n+1},i+1]-[x^{-p^{n+1}(p-1)}v_{n-1},i+1]\\
 \ & \ & \ \\
 \ & \in  & P^{\infty}([p])[x^{-p^n(p-1)}x^{(n+1)},i+1]+C+p\widehat{\mathcal B}=C+p\widehat{\mathcal B}.\end{array}$$  This ends our induction and our proof of surjectivity. 
 
 To conclude note that the morphism \ref{barbecue} is not an isomorphism. Indeed, the reduction modulo ${\mathbb K}^{\circ \circ}$ of the morphism
\begin{equation}
\label{turnoff}
P^{\infty}([p]):P^{\infty}(R[x,x^{-1}])^\circ\ra P^{\infty}(R[x,x^{-1}])^\circ\end{equation}
 is not injective because  the images of the classes
$[x',i]$ in $$P^{\infty}(R[x,x^{-1}])^\circ/{\mathbb K}^{\circ \circ}P^{\infty}(R[x,x^{-1}])^\circ$$ are non-zero and they belong to the kernel of the reduction modulo ${\mathbb K}^{\circ \circ}$ of the morphism \ref{turnoff}.
 \qed
 
\subsection{$\d$-characters and cocharacters}\label{mulcharcochar}
For ${\mathbb G}_m=\text{Spec}\ R[x,x^{-1}]$ and 
${\mathbb G}_a:=\text{Spec}\ R[z]$ recall from \cite[pg. 313]{char} the group homomorphism, called the {\it canonical $\d$-character},
$$\psi:J^1({\mathbb G}_m)\ra \widehat{{\mathbb G}_a},$$ 
defined by the $R$-algebra map 
sending
\begin{equation}
\label{polar0}
z\mapsto \Psi:=\frac{1}{p}\log\left(\frac{\phi(x)}{x^p}\right):=\sum_{n=1}^{\infty}(-1)^{n-1}
\frac{p^{n-1}}{n}\left(\frac{x'}{x^p}\right)^n\in R[x,x^{-1},x']^{\widehat{\ }}.\end{equation}
By Remark \ref{extensionn} we have an induced morphism
\begin{equation}
\label{polar1}
P^{\infty}(\psi):P^{\infty}({\mathbb G}_m)\ra P^{\infty}({\mathbb G}_a),\end{equation}
which we still call the {\it canonical $\d$-character}.

As a surprising turn of events, we will construct a ``canonical" section of $P^{\infty}(\psi)$. To explain our construction recall that
$$J^{\infty}(R[z])=R[z,z',z'',...]^{\widehat{\ }}$$
and for each $i\geq 1$ denote by $\phi^{-i}(z)$ the image of the element $[z,i]$ in 
$$C:=\varinjlim_{\phi} J^{\infty}(R[z]).$$ 
In particular $\phi(\phi^{-i}(z))=\phi^{-i+1}(z)$.
Now let $(a_n)_{n\in \bZ}$ be any sequence of elements $a_n\in pR$ such that
\begin{equation}
\label{priponel}
\lim_{n\ra \pm \infty} a_n=0,\end{equation}
Consider the element
\begin{equation}
\label{research1}
\Sigma= \exp\left( \sum_{n=1}^{\infty} a_n \phi^{-n}(z)\right)\in \widehat{C}
\end{equation}
Since $\widehat{C}$ possesses a natural $p$-derivation and $\sigma$ is invertible in
$\widehat{C}$
there is a unique $R$-algebra homomorphism
$$J^{\infty}(R[x,x^{-1}])\ra \widehat{C}$$
commuting with the natural $p$-derivations on the two rings and sending $x\mapsto \Sigma$.
Tensoring with ${\mathbb K}^\circ$ and taking the limits along the naturally induced relative Frobenius lifts we get a ring homomorphism, which we continue to denote by,
$$\sigma:P^\circ(J^{\infty}(R[x,x^{-1}]),\phi)\ra P^\circ(J^{\infty}(R[z]),\phi),$$
commuting with the relative Frobenius lifts, and sending $x\mapsto \Sigma$. This $\sigma$ induces a morphism, still denoted by,
\begin{equation}
\label{polar2}
\sigma:P^{\infty}({\mathbb G}_a)\ra P^{\infty}({\mathbb G}_m).\end{equation}
A map $\sigma$ as in \ref{polar2} will be called a {\it cocharacter} of ${\mathbb G}_m$.
We have the following result. 

\begin{theorem}\label{sandocal}
There is a unique 
 cocharacter $\sigma$, see Equation \ref{polar2}, with the property that $\sigma$
is a right inverse to the  canonical  $\d$-character $P^{\infty}(\psi)$ in Equation~\ref{polar1}. It is given by 
the sequence $a_n=p^n$ for $n\geq 1$ and $a_n=0$ for $n\leq 0$. \end{theorem}

First, we need a lemma. 

\begin{lemma}\label{polar59}
There exists a unique sequence $(a_n)_{n\in \bZ}$, $a_n\in R$, satisfying
\ref{priponel},
such that the following equality holds in $\widehat{C}$,
\begin{equation}\label{goigoigoila}
(\phi-p)\left( \sum_{n\in \bZ} a_n \phi^{-n}(z)\right)=pz.
\end{equation}
It is given by $a_n=p^n$ for $n\geq 1$ and $a_n=0$ for $n\leq 0$.
\end{lemma}

\begin{proof}
Equation \ref{goigoigoila} implies the conditions
\begin{equation}
\label{mamal}
\phi(a_n)=pa_{n-1}\ \ \ \text{for}\ \ \ n\in \bZ,\ \ n\neq 1.\end{equation}
If $a_{n_1}\neq 0$ for some $n_1\leq 0$ we get that $a_n$ does not converge to $0$ as $n\ra -\infty$, a contradiction. So $a_n=0$ for all $n\leq 0$. Then 
\ref{goigoigoila} implies $a_1=p$ and, with this condition, 
the recurrence relation Equation~\ref{mamal} has a unique solution, $a_n=p^n$ for $n\geq 1$.
\end{proof}

\

{\it Proof of Theorem \ref{sandocal}}.
Let $(a_n)_{n\in \bZ}$ with $a_n\in pR$ satisfy Equation \ref{priponel} and let $\Sigma$ be as in Equation \ref{research1}.
The composition
\begin{equation}
\label{polar7}
P^\circ(J^{\infty}(R[z]),\phi)\stackrel{z\mapsto \Psi}{\ra} P^\circ(J^{\infty}(R[x,x^{-1}]),\phi)\stackrel{x\mapsto \Sigma}{\ra} P^\circ(J^{\infty}(R[z]),\phi)\end{equation}
sends $z$ into
\begin{equation}
\label{zolof}
\frac{1}{p}\log\left(  \frac{\exp(\sum \phi(a_n)\phi^{-n+1}(z))}{(\exp(\sum a_n \phi^{-n}(z)))^p}
\right)= (\phi-p)\left( \sum_{n\in \bZ} a_n \phi^{-n}(z)\right).\end{equation}
If the composition in \ref{polar7} is the identity then, by Lemma \ref{polar59}, $a_n=p^n$ for $n\geq 1$ and $a_n=0$ for $n\leq 0$. Conversely, if $a_n=p^n$ for $n\geq 1$ and $a_n=0$ for $n\leq 0$ then  the map in \ref{polar7} sends $z\mapsto z$ and hence $z'\mapsto z'$, $z''\mapsto z''$, etc.
 By Proposition \ref{uniuni}
the composition \ref{polar7} is the identity. 
\qed

\begin{remark}\label{rmknosec}
One is tempted to expect that the  canonical  $\d$-character $P^{\infty}(\psi)$ in \ref{polar1}
has no section in the category of perfectoid spaces, or at least in the category ${\mathcal P}_{\Phi}$.
\end{remark}

\subsection{Complex analytic analogue}\label{complexGm}
The above construction can be viewed as an arithmetic analogue of the following complex analytic picture.
Consider the  differential equation
 \begin{equation}
 \label{exponential}
 \frac{\d_z u}{u}=\beta
 \end{equation}
 where $\beta$ a given function and $u$  an unknown function, both in the ring ${\mathcal O}$ of holomorphic functions of $z$ defined in a simply connected domain of ${\mathbb C}$, and  $\d_z:=d/dz$.
 Let $z_0$ be a point in the domain.
 Equation \ref{exponential} has the  solution $u=u_{\beta}$,
 \begin{equation}
 \label{expo}
 u_{\beta}(z):=\exp\left(\int_{z_0}^z \beta(t)dt\right).
 \end{equation}
 The map $\beta\mapsto u_{\beta}$ is a homomorphism.
As an arithmetic analogue of the above we may consider the group homomorphism $\psi:R^{\times}\ra R$ defined by
\begin{equation}
\label{psimap}
\psi(u) :=  \Psi_{|x=u,x'=\d u},
\end{equation}
where $\Psi$ is as in Equation \ref{polar0}.
 For any fixed $\beta\in R$, one can consider the {\it arithmetic differential equation}
\begin{equation}
\label{e}
\psi(u)=\beta.
\end{equation}
  Since $\psi$ in \ref{psimap} is a homomorphism it can be viewed as an analogue of the logarithmic derivative map 
  $${\mathcal R}^{\times}\ra {\mathcal R},\ \ \ u\mapsto \frac{\d_z u}{u};$$
   hence \ref{e} can be viewed as  an analogue of \ref{exponential}. 
Note that \ref{e} is equivalent to each of the following  equations
\begin{equation}\label{dodododo}
\phi (u)=\epsilon\cdot u^p\ \ \ \ \text{or}\ \ \ \ 
\d u=\alpha\cdot u^p
\end{equation}
where $\epsilon=\exp(p\beta)=1+p\alpha$ and $\exp$ the $p$-adic exponential. A solution to \ref{dodododo} (and hence to \ref{dodododo} or \ref{e}) has the form
\begin{equation}
\label{research}
u=\Sigma_{|\phi^{-i}(z)=\phi^{-1}(u)}\end{equation}
where $\Sigma$ is as in \ref{research1}.
So \ref{research} can be viewed as an analogue of \ref{expo}.

\section{The case of elliptic curves}

We now turn our attention to a detailed consideration of the case of elliptic curves. We give a couple of indications as to what results persist to more general abelian varieties. Our first point of interest concerns a study of $\d$-characters and cocharacters similar to Section~\ref{mulcharcochar}. 

\subsection{$\d$-characters and cocharacters}\label{deltacharacters}

Let us review here some of the theory in \cite{char, frob, equations}.
Assume $p\neq 2,3$.
Let $E$ be an elliptic curve over $R$. 
Recall from Proposition 4.45 in \cite{equations} that for $n\geq 1$ we have
$$\ker(J^n(E)\ra J^{n-1}(E))\simeq \widehat{{\mathbb G}_a},$$
where ${\mathbb G}_a=\Spec R[t]$ is the additive group scheme.

\

{\it Assume in what follows that
$E$ does not possess a Frobenius lift. }

\

The homomorphisms of groups in the category of $p$-formal schemes,
 $$\psi:J^2(E)\ra \widehat{{\mathbb G}_a}$$
  are called {\it $\d$-characters}. These objects were introduced in \cite{char}, where the more general case of abelian schemes was also considered, and may be viewed as  arithmetic analogues of the maps considered by Manin in the case of abelian varieties over function fields \cite{manin63}.  It is known that the $R$-module of $\d$-characters is  free of rank $1$ \cite[Prop. 3.2]{char}.   

Let now $T$ be an \'{e}tale coordinate on an open set $X=\text{Spec}\ A\subset E$ containing the origin and let $$\ell(T):=\ell_E(T)=\sum_{i=1}^{\infty} c_i T^i\in R[1/p][[T]],\ \  c_1=1,$$
be the formal group law on $E$. We also assume the origin of $E$ is defined in $X$ by the ideal $(T)$.
We have  natural embeddings $A\subset \widehat{A}\subset R[[T]]$.
Then $$\omega:=d \ell(T)=\Omega \cdot dT$$ satisfies 
$$\Omega:=\sum_{i=1}^{\infty} ic_iT^{i-1}\in A^{\times}$$
 and $\omega$ extends to an invertible $1$-form on the whole of $E$ which is a basis for the global $1$-forms on $E$. 
Recall from \cite[Prop 3.13]{equations} that for $n\geq 1$,
$$
\cO(J^n(X))=B_n:=\widehat{A}[T',...,T^{(n)}]^{\widehat{\ }}\subset S_n:= R[[T]][T',...,T^{(n)}]^{\widehat{\ }}$$
where  $T',T'',...$ are new variables and  $\d T=T'$, $\d T'=T''$, etc.
By \cite[Thm. 7.22]{equations} one can  choose a basis  $\psi=\psi_{E,\omega}$ for the module of $\d$-characters that, viewed as an element of 
$$\cO(J^2(E))\subset  B_2\subset S_2,$$
satisfies
\begin{equation}
\label{psi}\psi=\frac{1}{p}(
\ell(T)^{\phi^2}+\lambda_1\ell(T)^{\phi}+p\lambda_0\ell(T))\in S_2\end{equation}
where $\lambda_1:=\lambda_1(E,\omega)\in R$ and $\lambda_0:=\lambda_0(E,\omega)\in R$. 

We call this $\psi$ in \ref{psi} the {\it $\d$-character} attached to $E,\omega$. Let us record some facts proved in \cite{char, frob, equations} in the form of the following.

\begin{lemma}\label{contrast}\ 
\begin{enumerate}
\item Assume $E$  descends to $\bZ_p$ and $\omega$ is defined over $\bZ_p$. Then
$\lambda_1(E,\omega)= -a_p$, where $a_p$ is the trace of Frobenius on $\overline{E}$, and $\lambda_0(E,\omega)=1$. In particular $\lambda_1(E,\omega)\not\equiv 0 \bmod p$ if and only if $\overline{E}$ is ordinary.

\item If $\overline{E}$ is superordinary  then $\lambda_1(E,\omega)\not\equiv 0 \bmod p$. 

\item If $\overline{E}$ is supersingular then $\lambda_1(E,\omega)\equiv 0 \bmod p$. 
\end{enumerate}
\end{lemma}
\begin{proof}
 Assertion 1  follows from \cite[Sec 0.8 and 0.8]{frob}. Assertion 2  follows from \cite[Lem 4.2 and 4.4]{char} plus  \cite[Rmk 7.31]{equations}. Assertion 3 follows from \cite[Rmk. 7.31]{equations} plus \cite[Cor. 8.89]{equations}.
\end{proof}

\subsection{Sections of $P^{\infty}(\psi)$}

Consider again an elliptic curve $E$ over $R$ with no Frobenius lift and consider the $\d$-character
$$\psi=\psi_{E,\omega}:J^2(E)\ra \widehat{{\mathbb G}_a}$$  attached to some invertible $1$-form $\omega$ as in Equation \ref{psi}, where ${\mathbb G}_a=\Spec R[z]$.
By Remark \ref{extensionn} we have an induced morphism
\begin{equation}
\label{polar11}
P^{\infty}(\psi):P^{\infty}(E)\ra P^{\infty}({\mathbb G}_a)\end{equation}
which we still refer to as the  (perfectoid) {\it  $\d$-character}.
Let us say that $E$ is ordinary (respectively supersingular) if its reduction $\overline{E}$ is so. 
We will prove that if $E$ is supersingular then the map \ref{polar11} possesses a ``canonical 
section" (``canonical right inverse"),
\begin{equation}
\label{polar22}
\sigma:P^{\infty}({\mathbb G}_a)\ra P^{\infty}(E).\end{equation}
We will also prove that no such ``canonical section" exists for $E$ superordinary or ordinary and defined over $\bZ_p$. We need to introduce some terminology that makes the notion of ``canonical" we mean precise.

\

As in the case of the multiplicative group, recall that
$$J^{\infty}(R[z])=R[z,z',z'',...]^{\widehat{\ }}$$
and for each $i\geq 1$ denote by $\phi^{-i}(z)$ the image of the element $[z,i]$ in 
$$C:=\varinjlim_{\phi} J^{\infty}(R[z]).$$  
In particular, $\phi(\phi^{-i}(z))=\phi^{-i+1}(z)$. Consider any sequence $(a_n)_{n\in \bZ}$ of elements $a_n\in pR$ satisfying
$$\lim_{n\ra \pm \infty} a_n=0$$
and consider the element
\begin{equation}
\label{research11}
\Sigma=e_E\left( \sum_{n\in \Z} a_n \phi^{-n}(z)\right)\in p\widehat{C}
\end{equation}
where $e_E$ is the compositional inverse of the logarithm $\ell=\ell_E(T)$ of the formal group of $E$ with respect to an \'{e}tale coordinate $T$ on an affine open set  $X=\text{Spec}\ A\subset E$  containing the origin.
Since $\Sigma\in p\widehat{C}$ and  $\widehat{C}$ possesses a natural $p$-derivation there is a unique $R$-algebra homomorphism
\begin{equation}
\label{masa}
R[[T]][T',T'',T''',...]^{\widehat{\ }} \ra \widehat{C}\end{equation}
commuting with the natural $p$-derivations on the two rings and sending $T\mapsto \Sigma$.
Composing \ref{masa} with the natural ring homomorphism
$$J^{\infty}(A)\ra R[[T]][T',T'',T''',...]^{\widehat{\ }}$$
we get a homomorphism
$$J^{\infty}(A)\ra \widehat{C}.$$
Tensoring the latter with ${\mathbb K}^\circ$ and taking the limits along the naturally induced relative Frobenius lifts we get a ring homomorphism 
$$\sigma:P^\circ(J^{\infty}(A),\phi)\ra P^\circ(J^{\infty}(R[z]),\phi),$$
commuting with the relative Frobenius lifts, and sending $T\mapsto \Sigma$. This $\sigma$ induces a morphism, still denoted by, 
\begin{equation}
\label{polargoi}
\sigma:P^{\infty}({\mathbb G}_a)\ra P^{\infty}(X)\subset P^{\infty}(E).
\end{equation}
A morphism $\sigma$ as above will be called a {\it  cocharacter} of $E$. It depends on the sequence $(a_n)_{n\in \bZ}$ only,  not on the choice of \'etale coordinate $T$ nor the open set $X$.

\

Recall that we call $E$ {\it ordinary} (respectively {\it supersingular}) if it has ordinary (respectively supersingular) reduction mod $p$. Let us also say that $E$ is {\it superordinary} if it is ordinary  and has a Serre-Tate parameter $q(E) \not\equiv 1$ mod $p^2$.

\begin{theorem}\label{sandoval}
Assume $E$ has no Frobenius lift and let $P^{\infty}(\psi)$ be the  $\d$-character in \ref{polar11}. 
\begin{enumerate}
\item If $E$ is supersingular there exists a unique  cocharacter $\sigma$ with the property that $\sigma$ is  a right inverse to $P^{\infty}(\psi)$.
\item If $E$ is superordinary or ordinary and defined over $\bZ_p$ then   there exists no  cocharacter $\sigma$ with the property that $\sigma$ is  a right inverse to $P^{\infty}(\psi)$.
\end{enumerate}
\end{theorem}

\begin{remark} As in Remark~\ref{rmknosec}, one is again tempted to expect that the  canonical  $\d$-character $P^{\infty}(\psi)$ in \ref{polar11}
has no section in the category of perfectoid spaces, or at least in the category ${\mathcal P}_{\Phi}$.
\end{remark}

Before proving Theorem~\ref{sandoval}, we need a lemma characterizing when an element is in $pR$.  

\begin{lemma}\label{pppp}
Let  $\lambda_0, \lambda_1\in R$. The following are equivalent:

\begin{enumerate}
\item $\lambda_1\in pR$;

\item  There exists a sequence $(a_n)_{n\in \bZ}$, $a_n\in R$, such that
\begin{equation}
\label{pripon}
\lim_{n\ra \pm \infty} a_n=0,\end{equation}
and 
such that the following equality holds in $\widehat{C}$:
\begin{equation}\label{goigoigoi}
(\phi^2+\lambda_1\phi+p\lambda_0)\left( \sum_{n\in \bZ} a_n \phi^{-n}(z)\right)=pz.
\end{equation}
\end{enumerate}
Moreover if the above conditions hold then the sequence in condition 2 is unique, $a_2=p$,   $a_n=0$ for all $n\leq 1$, and $a_n\in pR$ for all $n\geq 3$.
\end{lemma}
\begin{proof}
First we make the following:

\medskip

{\it Claim}. If condition 2 holds then $a_n=0$ for all $n\leq 1$.

\medskip

Indeed, assume there exists $n_1\leq 1$ with $a_{n_1}\neq 0$ and seek a contradiction. We will recursively construct a sequence of integers 
$$1\geq n_1> n_2> n_3>  ...$$
such that 
$$\infty > v_p(a_{n_1}) \geq v_p(a_{n_2}) \geq v_p(a_{n_3}) \geq ...$$
This implies that $(a_n)_{n \in \mathbb{N}}$ does not converge $p$-adically to $0$ as $n\ra -\infty$, which is a contradiction. It is easy to find $n_1$ satisfying the condition. Assume now $n_1,...,n_k$ have been chosen. Picking out the coefficient of $\phi^{2-n_k}(z)$ in 
 Equation \ref{goigoigoi} we get
 \begin{equation}
 \label{fomica}
 \phi^2(a_{n_k})+\lambda_1\phi(a_{n_k-1})+p\lambda_0a_{n_k-2}=0.
 \end{equation}
 Now define $n_{k+1}$ to be 
 $$n_k-1\ \ \ \text{or}\ \ \ n_k-2$$
  according as 
 $$v_p(a_{n_k-1})\leq  v_p(a_{n_k})\ \ \ \text{or}\ \ \ v_p(a_{n_k-1})>  v_p(a_{n_k}),$$
  respectively.
 We claim that  
 $$v_p(a_{n_{k+1}})\leq v_p(a_{n_k}).$$ This is clear in case
 $v_p(a_{n_k-1})\leq  v_p(a_{n_k})$. 
  In the case $v_p(a_{n_k-1})>  v_p(a_{n_k})$,
 we have
 $$v_p(a_{n_{k+1}})=v_p(a_{n_k-2})< v_p(p\lambda_0a_{n_k-2})=v_p(\phi^2(a_{n_k})+\lambda_1\phi(a_{n_k-1}))=v_p(a_{n_k}),$$
 which ends the proof of the claim. 
 
 \

 Now assume condition 2 holds and let us prove condition 1 holds. Assume $\lambda\in R^{\times}$ and seek a contradiction. By the claim above $a_n=0$ for $n\leq 1$ so condition \ref{goigoigoi} is equivalent to the conditions  
\begin{equation}
\label{can}
a_2=p,\ \ a_3=-p\phi^{-2}(\lambda_1),\end{equation}
together with the following conditions valid for $n\geq 2$:
\begin{equation}\label{recurrence}
\phi^2(a_{n+2})+\lambda_1\phi(a_{n+1})+p\lambda_0a_n=0.
\end{equation}
Now the recurrence relation \ref{recurrence} with conditions \ref{can} can be trivially and uniquely solved for $a_n\in R$, $n\geq 2$. 
Using the fact that $\lambda_1\in R^{\times}$ we easily prove by induction that $v_p(a_n)=1$ for all $n\geq 2$, a contradiction.

\

Conversely let us assume that condition 1 holds and let us prove that condition 2 holds.
Setting  $a_n=0$ for $n\leq 1$ we have that Equation \ref{goigoigoi} holds, again,  equivalent to establishing that 
Equation \ref{can} and Equation \ref{recurrence} also hold. Again, the recurrence relation \ref{recurrence} with conditions \ref{can} can be trivially solved for $a_n\in R$, $n\geq 2$. 
Using the fact that $\lambda_1\in pR^{\times}$ we easily prove by induction that 
$$a_{2k-1}, a_{2k}\in p^kR, \ \ \ k\geq 2,$$
hence $a_n\ra 0$, $p$-adically, so condition 2 holds. 

\

Finally assuming that the (equivalent) conditions 1 and 2 hold we need to prove uniqueness of the sequence $(a_n)$. By the Claim above we have $a_n=0$ for $n\leq 1$. Then, again,  \ref{goigoigoi} reduces to the recurrence relation \ref{recurrence} with conditions \ref{can}, which  has a unique solution with $a_n\in R$, $n\geq 2$. 

\end{proof}

We now return to the proof of the theorem. 

\

{\it Proof of Theorem \ref{sandoval}}.
Assume first that $\overline{E}$ is supersingular. By  Lemma \ref{contrast}, 
the coefficient $\lambda_1$ in 
Equation \ref{psi} satisfies $\lambda_1\in pR$. 
So, by Lemma \ref{pppp}, there exists a sequence $(a_n)_{n\in \bZ}$ such that for $\Sigma$
as defined in Equation \ref{research11} we have that Equation \ref{goigoigoi} holds. It follows  that the composition
\begin{equation}\label{lala}
P^\circ(J^{\infty}(R[z]),\phi)\stackrel{z\mapsto \psi}{\ra} P^\circ(J^{\infty}(A),\phi)\stackrel{T\mapsto \Sigma}{\ra} P^\circ(J^{\infty}(R[z]),\phi)\end{equation}
sends $z\mapsto z$ and hence $z'\mapsto z'$, $z''\mapsto z''$, etc. By Proposition \ref{uniuni}
the composition \ref{polar7} is the identity; hence the existence part of assertion 1 of the Theorem follows. The uniqueness part follows because if $(a'_n)_{n\in \bZ}$ is another sequence for which the corresponding composition \ref{lala}, with $\Sigma$ replaced by the corresponding $\Sigma'$ attached to $(a'_n)_{n\in \bZ}$, 
 is the identity then the sequence $(a'_n)$ coincides with $(a_n)$ by the uniqueness property in Lemma \ref{contrast}. 
 
 Finally assume that $E$ is superordinary or ordinary and defined over $\bZ_p$. Then, by  Lemma \ref{contrast}, 
the coefficient $\lambda_1$ in 
equation \ref{psi} satisfies $\lambda_1\in R^{\times}$. If a 
  cocharacter exists that is a right inverse to $P^{\infty}(\psi)$ then we get that the composition \ref{lala} is the identity for some sequence $(a_n)_{n\in \bZ}$ satisfying 
\ref{pripon} and \ref{goigoigoi}. This will force a contradiction with Corollary \ref{ppp}, which while proved later, is logically proved independent of the results of this section. We postpone the proof of Corollary~\ref{ppp} until Section~\ref{sec:QL}. This ends our proof.\qed

\begin{remark}
In this section we investigated the $\d$-characters $\psi$ of elliptic curves with no Frobenius lifts. One can ask for an analysis of the characters $\psi$ for elliptic curves $R$ that are canonical lifts. The analysis in this case is much simpler, it resembles the case of ${\mathbb G}_m$,  and we leave it to the reader.
\end{remark}

\subsection{Complex analytic analogue}\label{complexelliptic}
The arithmetic theory above is an analogue of a classical situation which we describe in what follows.
Let $E$ be the smooth projective elliptic curve over ${\mathbb C}(z)$ defined by the equation
$$y^2=x(x-1)(x-z).$$
 Let ${\mathcal R}$ be the ring of holomorphic functions on a simply connected domain not containing $0,1$. By Fuchs and Manin \cite{manin63},  there is a remarkable group homomorphism 
\begin{equation}
\label{maninsfunct}
\psi:E({\mathcal R})\ra {\mathcal R}\end{equation}
 defined as follows. If $P=(X,Y)\in E({\mathcal R})$ is a point of $E$ with $X,Y\in {\mathcal R}$, i.e., $x(P)=X$, $y(P)=Y$, 
then
$$
\psi(P) = \Lambda\ 
\int_{\infty}^{(X,Y)}\frac{dx}{y}=F(X,Y,\d_zX,\d_z^2 X),$$
where $\Lambda$ is the linear differential operator
$$\Lambda:=z(1-z)\left(z(1-z)\d_z^2+(1-2z)\d_z-\frac{1}{4}\right),$$ where $\d_z := d/dz$. 
Thus $F$ is a certain rational function with coefficients in ${\mathbb C}(z)$ in $4$ variables.
The expression involving the integral is well defined because the linear differential operator $\Lambda$ annihilates any function of $z$ of the form
$\int_{\gamma}\frac{dx}{y}$ where $\gamma$ is an integral $1$- cycle.
 Functions of the form $\int_{\gamma}\frac{dx}{y}$ are called in \cite{manin63} {\it period-functions}. 
 In this general case,  for $P=(X,Y)$ with $\psi(P)=0$, the {\it multivalued} integral $\int_{\infty}^P\frac{dx}{y}$, as a function of $z$, is, for each of its branches, a ${\mathbb C}$-linear combination of period-functions, so
  \begin{equation}
 \label{P}
 P=\pi\left(a_1 \cdot \int_{\gamma_1}\frac{dx}{y}+a_2\cdot \int_{\gamma_2}\frac{dx}{y}\right),\end{equation}
 where $\pi$ is the uniformization map for our family of elliptic curves the inverse of the multivalued Abel-Jacobi map $Q\mapsto \int_{\infty}^Q\frac{dx}{y}$. Here $a_1,a_2\in {\mathbb C}$ and $\gamma_1,\gamma_2$ is a basis for the integral homology. Of course $\pi$ is defined by elliptic functions. Our function $\psi$ in \ref{psi} is an arithmetic analogue of $\psi$ in Equation \ref{maninsfunct}. The ``points" of our arithmetic Manin kernel, to be discussed later, are an arithmetic analogue of the points \ref{P}. Finally if $\beta\in {\mathcal R}$ and $I(\beta)\in {\mathcal R}$
 is a solution of the linear inhomogeneous differential equation
 $$\Lambda (I(\beta))=\beta$$
 normalized such that its value and the value of its derivative at some fixed point vanish then
 a solution $P_{\beta}$ of $\psi(P_{\beta})=\beta$ is given by 
 \begin{equation}
 P_{\beta}=\pi(I(\beta)).
 \end{equation}
 The map $\sigma$ in \ref{polargoi} is an arithmetic analogue of the map $\beta\mapsto \pi(I(\beta))$.

 \section{Perfectoid spaces attached to arithmetic differential equations}
 
 \subsection{Arithmetic differential equations}
  Assume in what follows that $X$ is a smooth scheme  over $R$ and 
 set
 $$X^n:=J^n(X)$$
  for $n\geq 0$. 
  
  \begin{definition}
  An {\it arithmetic differential equation} on $X$, of order $n$, is a 
  closed formal subscheme
  $Y_n\subset X^n$.\end{definition}

  Let ${\mathcal I}$ be the ideal sheaf in $\cO_{X^n}$
  defining a $Y_n$ as above and for any $j\geq 1$ let
  $$Y_{n+j}\subset X^{n+j}$$
  be the closed formal subscheme with ideal generated by
  $${\mathcal I},\d{\mathcal I},...,\d^j{\mathcal I}.$$
   Also set
   $Y_m=X^m$ for $0\leq m\leq n-1$.
  We have a projective system of formal schemes
$$...\ra Y_3 \ra Y_2\ra Y_1\ra Y_0,$$
which we refer to as the {\it prolongation} of $Y_n$.
There are naturally induced $p$-derivations
\begin{equation}
\label{xz}
\pi_{*}\cO_{Y_0}\stackrel{\d}{\ra} \pi_{*}\cO_{Y_1}\stackrel{\d}{\ra} \pi_{*}\cO_{Y_2}\stackrel{\d}{\ra} \pi_{*}\cO_{Y_3} \ra ...\end{equation}
where $\pi:Y_n\ra X^0=\widehat{X}$ are the natural projections.
Define
$$\cO_{Y_{\infty}}:=\varinjlim \pi_{*}\cO_{Y_n}.$$
Recall that, since $X$ is Noetherian, the above limit in the category of sheaves equals the limit in the category of pre-sheaves. The system \ref{xz} induces a $p$-derivation $\d$ and an attached Frobenius lift $\phi$,
$$\d,\phi:\cO_{Y_{\infty}}\ra \cO_{Y_{\infty}}.$$

 \subsection{Quasi-linearity}\label{Qsection}
 Let $X$ be a smooth affine curve over $R$, let $A=\cO(X)$, and $B_n=\cO(J^n(X))$.
 Let $s\geq 1$ and $r\geq 0$. Set $X\ra {\mathbb A}^1=\text{Spec}\ R[T]$ an \'{e}tale coordinate, and view $T\in A$. Recall that we have a canonical identification 
 \begin{equation}
 \label{Bn}
 B_n=\widehat{A}[T',...,T^{(n)}]^{\widehat{\ }}.\end{equation}
 For $0\leq i\leq n$ define the subrings of $B_n$ by
 \begin{equation}
 \label{Fib}
 F^iB_n:=\sum_{k=0}^{n-i}p^k B_{i+k}=B_i+...+p^{n-i}B_n.\end{equation}
 For $0\leq i\leq n-1$ we have
 $$F^iB_n\subset F^{i+1}B_n.$$
 Also, for $0\leq i \leq n$ we have
 $$F^iB_n\subset F^iB_{n+1}.$$
 Furthermore, for $0\leq i\leq n$, we have:
 \begin{equation}
 \label{inclusions}
 \phi(F^iB_n)\subset F^iB_{n+1},\ \ \ \d(F^iB_n)\subset F^{i+1}B_{n+1}.\end{equation}
 
 \begin{definition}\label{rs}
 An element $f\in  B_{r+s}$ is called  {\it quasi-linear} of order $(r,s)$ if there exists an \'{e}tale coordinate $T\in A$ and elements $a_0,...,a_r\in \widehat{A}$ and $b_{s-1}\in F^{s-1}B_{r+s}$
 such that $a_r\in \widehat{A}^{\times}$ and
\begin{equation}
\label{quasi}
f=\sum_{i=0}^r a_i (T^{(s)})^{\phi^i}+b_{s-1}.
\end{equation}
We say $f$ is {\it non-degenerate} (respectively {\it totally degenerate}) if 
$a_0\in \widehat{A}^{\times}$ (respectively $a_0\in p\widehat{A}$).
 \end{definition}
 
 \begin{remark}\label{fifi}
 Clearly if $f$ is  quasi-linear of order $(r,s)$ and $g\in F^{s-1}B_{r+s}$ then $f+g$ is  quasi-linear of order $(r,s)$. If $f$ is  non-degenerate (respectively  totally degenerate) then so is $f+g$.\end{remark}

 \begin{definition}\ 
  Let $X$ be a smooth curve over $R$ and $Y_{r+s}\subset J^{r+s}(X)$
 an arithmetic differential equation of order $r+s$. We say $Y$ is {\it quasi-linear} of order $(r,s)$
 if there is a Zariski  open cover
 $$X=\bigcup_i X_i$$
 such that for all $j$ the ideal of $Y_{r+s}\cap J^{r+s}(X_i)$ in $\cO(J^{r+s}(X_i))$ is generated by
 a quasi-linear element $f_i$ of order $(r,s)$.  We say $Y$ is {\it non-degenerate} (respectively {\it totally degenerate}) if, in addition,  all $f_i$ can be chosen to be non-degenerate (respecively totally degenerate).
 \end{definition}
 
 \begin{remark}
 If $X$ is a scheme of finite type over $R$, $X'\ra X$ is a morphism of finite type, and $Y_n\subset J^n(X)$ is
 an arithmetic differential equation of order $n$ then we may define the {\it pull-back} $Y'_n$ of $Y_n$
 to be the arithmetic differential equation  on $X'$ given by $Y'_n:=Y_n\times_{J^n(X)}J^n(X')$.
 Clearly, if $X$ is a smooth curve and $X'\ra X$ is \'{e}tale then $Y'_n$ is quasi-linear (respectively non-degenerate, totally degenerate) provided $Y_n$ is quasi-linear (respectively non-degenerate, totally degenerate).
 \end{remark}
 
 \begin{lemma}
 \label{heredity}
 Assume $f$ is quasi-linear of order $(r,s)$ as in Definition \ref{rs}, given by \ref{quasi}. Then, for $j\geq 1$,  $\d^j f$ is quasi-linear
 of order $(r,s+j)$ given by
 \begin{equation}
\label{quasiii}
\d^j f=f_j:=\sum_{i=0}^r a_i^{p^j}(T^{(s+j)})^{\phi^i}+b_{s+j-1},
\end{equation}
for some $b_{s+j-1}\in F^{s+j-1}B_{r+s+j}$. \end{lemma}
 
 \begin{proof}
 It is enough to check the above for $j=1$. Since $T^{(s)}\in B_s=F^sB_s$, by \ref{inclusions}, we have
 \begin{equation}
 \label{voce}
 (T^{(s)})^{\phi^i}\in F^sB_{s+i}
 \end{equation}
  for $0\leq i\leq r$ so all the terms in the right hand side of \ref{quasi} are in $F^sB_{s+r}$. In particular
 $$\d f\in \sum_{i=0}^r \d(a_i(T^{(s)})^{\phi^i}) +\d b_{s-1}+F^sB_{r+s}.$$
 Note that, by \ref{inclusions}, $\d b_{s-1}\in  F^sB_{r+s+1}$. On the other hand, using \ref{voce} and the fact that $\phi$ commutes with $\d$, we have:
 $$
 \begin{array}{rcl}
 \d(a_i(T^{(s)})^{\phi^i}) & = & a_i^p\cdot  \d((T^{(s)})^{\phi^i})+\d a_i\cdot  (T^{(s)})^{\phi^{i+1}}\\
 \ & \ & \ \\
 \ & \in & a_i^p\cdot (T^{(s+1)})^{\phi^i}+F^sB_{s+r+1},
 \end{array}
 $$
 which concludes the proof.
 \end{proof}
 
 \subsection{Prolongations of quasi-linear equations}
 Assume in what follows that $X$ is a smooth curve over $R$ and, as usual, 
 set $X^n:=J^n(X)$
  for $n\geq 0$. For simplicity, and in view of our applications, we will consider here an arithmetic differential equation
  $$Y_2\subset X^2$$
  which is quasi-linear of order $(1,1)$. Most of our analysis can be
   easily extended to the $(r,s)$ case.
  Recall, we defined the following sequence, which formally is a prolongation of $Y_2$,
$$...\ra Y_3 \ra Y_2\ra Y_1\ra Y_0$$
and induced $p$-derivations \ref{xz}
inducing a $p$-derivation $\d$ on 
$$\cO_{Y_{\infty}}:=\varinjlim \pi_{*}\cO_{Y_n}.$$
As usual we let 
$$\overline{Y_n}=Y_n\otimes_R k.$$
For $j\geq -2$ we let $\overline{Z}_{j+2}$ be the {\it closed scheme theoretic image} of the monomorphism
$$\eta:\overline{Y_{j+3}}\ra \overline{Y_{j+2}}.$$
That is we may view $\overline{Z}_{j+2}$ as the closed subscheme of $\overline{Y_{j+2}}$ defined by the ideal
$$\ker(\cO_{\overline{Y_{j+2}}}\ra \eta_*\cO_{\overline{Y_{j+3}}}).$$
 We get natural morphisms of schemes over $k$,
$$
\begin{array}{cccccccc}
 \ra & \overline{Z}_3 & \ra   & \overline{Z}_2  & \ra & \overline{Z}_1 & \ra & \overline{Z}_0\\
 \nearrow & \downarrow & \nearrow  &  \downarrow & \nearrow & \downarrow & \nearrow & || \\
 \ra & \overline{Y_3} & \ra   & \overline{Y_2}  & \ra & \overline{Y_1} & \ra & \overline{Y_0}\\
 \  & \downarrow & \   &  \downarrow & \  & || & \ & ||\\
 \ra & \overline{X^3} & \ra & \overline{X^2}  & \ra & \overline{X^1} & \ra & \overline{X^0}
\end{array}
$$
where the vertical arrows are closed immersions. Our next main object is to prove the following theorem. Recall, a map is called {\it purely inseparable} of the relative K\"{a}hler differentials vanish.
 
\begin{theorem}\label{thm1}

Assume $Y_2\subset X^2$ is a quasi-linear arithmetic differential equation of order $(1,1)$ on a smooth curve $X$ over $R$. The following hold for $j\geq -2$.

\begin{enumerate}
\item The formal schemes $Y_{j+2}$ are flat over $R$.
\item The maps $\overline{Z}_{j+3}\ra \overline{Z}_{j+2}$ are finite and flat of degree $p$.
\item $\overline{Z}_{j+3}$ is a Cartier divisor on $\overline{Y_{j+3}}$.
\item $\overline{Y_{j+3}}\simeq \overline{Z}_{j+2}\times_{\overline{X^{j+2}}}\overline{X^{j+3}}$ over 
$\overline{Z}_{j+2}$.
\item If $Y_2$ is non-degenerate the maps $\overline{Z}_{j+3}\ra \overline{Z}_{j+2}$  are \'{e}tale.
\item If $Y_2$ is non-degenerate the formal schemes $Y_{j+2}$ are $p$-smooth.
\item If $Y_2$ is totally degenerate the maps $\overline{Z}_{j+3}\ra \overline{Z}_{j+2}$ are purely inseparable.
\end{enumerate}
\end{theorem}

Before turning to the proof explore some consequences.

\subsection{Perfectoid consequences}

Later in the section, we will prove a proof of the following theorem.

\begin{theorem} \label{popi}
Assume $Y_2\subset X^2$ is a non-degenerate quasi-linear arithmetic differential equation 
of order $(1,1)$ on a smooth curve $X$ over $R$ and let $X=\cup X_i$ be an open affine cover. Let $C_{\infty,i}:=H^0(X_i,\cO_{Y_{\infty}})$. One has that
\begin{enumerate}
\item $C_{\infty,i}$ are $p$-torsion free. In particular, $\widehat{C_{\infty,i}}$ are $p$-torsion free.

\item $\overline{C_{\infty,i}}$  are 
integral and ind-\'{e}tale over $\cO(\overline{X_i})$. In particular, they are reduced
 and integrally closed in their total ring of fractions.
 \end{enumerate}
\end{theorem}

Note also that $C_{\infty,i}$ and hence $\widehat{C_{\infty,i}}$  carry (compatible) Frobenius lifts $\phi$; coming from the universal $p$-derivations. Recall that for a scheme with a  principal cover $X=\cup X_i$  there exist
$s_{ij}\in \cO(X_i)$ such that 
$\cO(X_i\cap X_j)=\cO(X_i)_{s_{ij}}$.

Combining Theorems \ref{r1} and \ref{popi}, plus Remark \ref{d1}, we have the following consequence. 

\begin{corollary}\label{argue}
Let $Y_2\subset X^2$ be as in Theorem \ref{popi} and let $X=\cup X_i$ be a principal  cover as in \ref{principalcover}. 
The ${\mathbb K}$-algebras $$P_i:=P(\widehat{C_{\infty,i}},\phi)$$ are perfectoid, equipped with relative Frobenius lifts  and we have isomorphisms
$$((P^\circ_i)_{s_{ij}})^{\widehat{\ }}\simeq ((P^\circ_j)_{s_{ji}})^{\widehat{\ }}$$
compatible with the relative Frobenius lifts and 
satisfying the obvious cocycle condition. Moreover, for each $i$,
$P^\circ_i/{\mathbb K}^{\circ \circ}P^\circ_i$
is the perfection of an integral ind-\'{e}tale algebra
  over $\cO(\overline{X_i})$.
In particular the perfectoid spaces $\Spa(P_i,P_i^\circ)$
glue together to give a perfectoid space $P^{\infty}(Y_2)$ equipped with an invertiblle relative Frobenius lift plus a closed immersion  $P^{\infty}(Y_2)\ra P^{\infty}(X)$ compatible with the relative Frobenius lifts. Furthermore the schemes
$\Spec P_i^\circ/{\mathbb K}^{\circ \circ}P_i^\circ$ glue together to give a scheme 
$P^{\infty}(Y_2)_k$ over $k$
which is the perfection of a profinite pro-\'{e}tale cover of $\overline{X}$.
\end{corollary}

The scheme $P^{\infty}(Y_2)_k$ is, of course,  the reduction mod ${\mathbb K}^{\circ \circ}$ of $P^{\infty}(Y_2)$ defined by the induced principal cover of the latter.

\subsection{Proofs} For the proof of Theorem \ref{thm1} we may assume $X$ is affine
and possesses an \'{e}tale coordinate $T\in A:=\cO(X)$. We assume the notation in subsection \ref{Qsection}. In particular $B_n$ is given by Equation \ref{Bn}. Let $f_j$ be as in Equation \ref{quasi} and 
let $\overline{f_j}\in \overline{A}[T',...,T^{(j+2)}]$ be its image. Then
$$\overline{f_j}= \overline{a_1}^{p^j}\cdot (T^{(j+1)})^p+
\overline{a_0}^{p^j}\cdot T^{(j+1)}
+\overline{b_j},\ \ \overline{b_j}\in \overline{A}[T',...,T^{(j)}],$$
so $\overline{f_j}$ is, up to a multiplication by an invertible element $\overline{a_1}$ in $\overline{A}$, a monic polynomial in $T^{(j+1)}$ of degree $p$ with coefficients in
$\overline{A}[T',...,T^{(j)}]$.
Moreover
\begin{equation}\label{dd}
\frac{d\overline{f_j}}{d T^{(j+1)}}=\overline{a_0}.\end{equation}
Setting
$$C_n=\cO(Y_n),\ \ \overline{D}_n:=\cO(\overline{Z}_n),$$
we have, for $j\geq 0$,
\begin{equation}
\label{docolo}
C_{j+2}= \frac{\widehat{A}[T',...,T^{(j+2)}]^{\widehat{\ }}}{(f_0,...,f_j)},\end{equation}
Also
$$\overline{C_1} = \overline{B_1}=\overline{A}[T'],\ \ \overline{C_0}=\overline{B_0}=\overline{D_0}=\overline{A}.$$
We conclude the following. 

\begin{lemma} \label{snapper} \ 

For $j\geq 0$ we have:
$$
\overline{C_{j+2}} = \frac{\overline{A}[T',...,T^{(j+1)}]}{(\overline{f_0},...,\overline{f_j})}
[T^{(j+2)}].$$ 

For $j\geq -1$ we have
$$\overline{D_{j+2}}=\text{Im}(\overline{C_{j+2}} \ra \overline{C_{j+3}})=
\frac{\overline{A}[T',...,T^{(j+2)}]}{(\overline{f_0},...,\overline{f_{j+1}})}.$$ 
\end{lemma}

It is also useful to record the following lemma. 

\begin{lemma}\label{wither}\

\begin{enumerate}
\item For any ring $B$ and any monic polynomial $f\in B[T]$ in one variable $T$, $f$ is a non-zero-divisor in $B[T]$; also any non-zero-divisor in $B$ is a non-zero-divisor in $B[T]$.

\item Assume $B$ is a Noetherian $p$-adically complete ring and let $p,b_0,...,b_n$ be a regular sequence in $B$. Then the ring 
$B/(b_0,...,b_n)$ is $p$-torsion free.
\end{enumerate}
\end{lemma}

\begin{proof}
Assertion 1 is trivial. To check assertion 2 it is enough to show that the  localization of $B/(b_0,...,b_n)$ at any maximal ideal of this ring is $p$-torsion free. Let $M/(b_0,...,b_n)$ be such a maximal ideal, where $M$ is a maximal ideal of $B$ containing $b_0,...,b_n$. Since $B$ is $p$-adically complete we also have $p\in M$; cf. \cite{matsumura}, p. 57. 
By flatness, the images in $B_M$ of $p,b_0,...,b_n$ form a regular sequence in $B_M$. Since $B_M$ is Noetherian and local  the images in $B_M$ of $b_0,...,b_n,p$  still form a regular sequence; cf. \cite[pg. 127]{matsumura}. This ends our proof.
\end{proof}

We are ready to give the proof of Theorem~\ref{thm1}. 

\bigskip

{\it Proof of Theorem \ref{thm1}}.
For assertion 1 it is enough to check that the ring $$C_{j+2}= \frac{\widehat{A}[T',...,T^{(j+2)}]^{\widehat{\ }}}{(f_0,...,f_j)}$$ is $p$-torsion free. Since $\widehat{A}[T',...,T^{(j+2)}]^{\widehat{\ }}$ is $p$-adically complete and Noetherian, by Lemma \ref{wither}, assertion 2, and by Lemma \ref{wither}, it is enough to check that $\overline{f_0},...,\overline{f_j}$ form a regular sequence in $\overline{A}[T',...,T^{(j+2)}]$. But  $\overline{f_i}$ are monic polynomials
in $T^{(i+1)}$ with coefficients in $\overline{A}[T',...,T^{(i)}]$ and we may conclude by assertion 1 in Lemma \ref{wither}.
Assertion 2, 3, 4, 5, and 7 are clear from Lemma \ref{snapper} plus assertion 1 of Lemma \ref{wither}.
To check assertion 6 it is sufficient to check that $C_{j+2}/p^{\nu}C_{j+2}$ is smooth over 
$R/p^{\nu}R$ for all $\nu$. We know $C_{j+2}/p^{\nu}C_{j+2}$ is finitely generated and flat over 
$R/p^{\nu}R$ by assertion 1; so it is enough to show that $C_{j+2}/pC_{j+2}$ is smooth over 
$R/pR$, which follows, again, from  Lemma \ref{snapper}.
\qed

\bigskip

{\it Proof of Theorem \ref{popi}}. We may assume  $X=X_i$ and set $C_{\infty,i}=C_{\infty}$.
With $C_n=\cO(Y_n)$ and $\overline{D}_n=\cO(\overline{Z}_n)$ as in the proof of Theorem \ref{thm1} one has
$$C_{\infty}=\varinjlim C_n.$$
Since, by Theorem \ref{thm1}, $C_n$ are $p$-torsion free, $C_{\infty}$ is $p$-torsion free which proves assertion 1.
On the other hand,
$$\overline{C_{\infty}}=\varinjlim \overline{C_n}=\varinjlim \overline{D}_n.$$
By Theorem \ref{thm1},
 $\overline{D}_n$ are finite \'{e}tale over $\overline{D}_{n-1}$ so 
 $\overline{C_{\infty}}$ is integral and ind-\'{e}tale over $\cO(\overline{X_i})$.
 This ends the proof of assertion 2.
\qed

\section {Perfectoid spaces attached to arithmetic Manin kernels}

\subsection{Arithmetic Manin kernel}
Recall the terminology and notation from subsection \ref{deltacharacters}. In particular, let $E$ be an elliptic curve over $R$ not possessing a Frobenius lift and let $\psi$ be a basis  for the module of $\d$-characters.

\begin{definition}
The {\it arithmetic Manin kernel} is the arithmetic differential equation
$$Y_2:=Y_2(E):=\ker  \psi\subset J^2(E).$$\end{definition}

Clearly $Y_2(E)$ only depends on $E$ and not on the basis $\psi$.
By the universality property of the $p$-jet spaces one gets  canonical homomorphisms
$$\psi^j:J^{j+2}(E)\ra J^j({\mathbb G}_a)$$
for all $n\geq 0$, commuting in the obvious sense with the universal $p$-derivations.
The prolongation 
$$...\ra Y_j\ra ... \ra Y_2 \ra Y_1\ra Y_0$$
of $Y_2$ satisfies of course, 
$$Y_{j+2}=\ker \psi^j,\ \ \ \text{for}\ \ \ j\geq 0$$
hence they are groups in the category of $p$-adic formal schemes.
As before set $\overline{Z}_j$ be the closed scheme theoretic image of
$\overline{Y_{j+1}}\ra \overline{Y_j}$.
It is a general fact that scheme theoretic images of homomorphisms of group schemes of finite type over a field are closed subgroup schemes; in particular  $\overline{Z}_j$ are closed subgroup schemes of
$\overline{Y_j}$. 

\subsection{Perfectoid consequence} Our next aim is to prove the following theorem. 

\begin{theorem}\label{elli}
Let $E$ be an elliptic curve over $R$ not possessing a Frobenius lift and let $Y_2\subset J^2(E)$ be the arithmetic Manin kernel. 
\begin{enumerate}
\item The space $Y_2$ is quasi-linear of order $(1,1)$.

\item Assume either $E$ is superordinary or $E$ is ordinary and defined over $\bZ_p$. Then $Y_2$ is non-degenerate.

\item Assume $E$ is supersingular. Then $Y_2$ is totally degenerate.
\end{enumerate}
\end{theorem}

Theorem \ref{elli} and Corollary \ref{argue} imply the following result.

\begin{corollary}\label{lll}
Let $E$ be an elliptic curve over $R$, not possessing a Frobenius lift, 
and assume either $E$ is superordinary or $E$ is ordinary and defined over $\bZ_p$.
Let $Y_2=Y_2(E)\subset J^2(E)$ be the arithmetic Manin kernel and let $X=\cup X_i$ be a principal   cover. Let $C_{\infty,i}:=H^0(X_i,\cO_{Y_{\infty}})$.
Then the $K$-algebras $P_i:=P(\widehat{C_{\infty,i}},\phi)$ are perfectoid, equipped with relative Frobenius lifts  and we have isomorphisms
$$((P^\circ_i)_{s_{ij}})^{\widehat{\ }}\simeq ((P^\circ_j)_{s_{ji}})^{\widehat{\ }}$$
compatible with the relative Frobenius lifts and 
satisfying the  cocycle condition. Moreover, for each $i$,  the ring
$P^\circ_i/{\mathbb K}^{\circ \circ}P^\circ_i$ is the perfection of an integral ind-\'{e}tale algebra over
 $\cO(\overline{X_i})$. In particular the perfectoid spaces $\Spa(P_i,P_i^\circ)$
glue together to give a perfectoid space $P^{\infty}(Y_2)$ equipped with an  invertible relative Frobenius lift plus a closed immersion $P^{\infty}(Y_2)\ra P^{\infty}(E)$ compatible with the relative Frobenius lifts. 
 Furthermore the schemes
$\Spec P_i^\circ/{\mathbb K}^{\circ \circ}P_i^\circ$ glue together to give a scheme $P^{\infty}(Y_2)_k$
over $k$ which is the perfection of a profinite pro-\'{e}tale cover of $\overline{E}$.
The above construction is functorial with respect to  \'{e}tale isogenies of degree not divisible by $p$ and also with respect to translations by torsion points.
\end{corollary}

In order to obtain a functor one needs, as before, to attach to any $E$ a principal basis.
The scheme $P^{\infty}(Y_2)_k$ is the reduction mod ${\mathbb K}^{\circ \circ}$ of $P^{\infty}(Y_2)$ defined by the induced principal cover of the latter.
Theorem \ref{hoi} follows now from Corollary \ref{lll} and from our construction: one 
simply needs to take $P^{\sharp}(E)$ in Theorem \ref{hoi} to be equal to $P^{\infty}(Y_2(E))$ in Corollary \ref{lll}.

\begin{remark}
Let us record the following variant of Theorem \ref{elli}. Assume $E$ is an elliptic curve over $R$ not possessing a Frobenius lift, let 
$X\subset E$ be an open subset, let $g\in F^0B_2$, let 
$f=\psi+g$
and let $Y_2(f)\subset J^2(X)$ be the arithmetic differential equation defined by $f$.
Then,  by Remark \ref{fifi}, $Y_2(f)$ is quasi-linear of order $(1,1)$. If in addition 
 $E$ is either superordinary or  ordinary and defined over $\bZ_p$ then $Y_2(f)$ is non-degenerate. Examples of such $f$'s are given by the {\it arithmetic Painlev\'{e} equations}
 in \cite{BYM} where $g=\phi(u)$ for some $u\in \cO(X)$. Hence for such $f$'s the obvious analogue of Corollary \ref{lll} holds.
\end{remark}

We proceed now with preparations for the proof of Theorem \ref{elli}. Note that the curve $E$ can be covered by translates $$X+Q$$ where $Q$ are torsion points in $E(R)$ and  note that $\psi$ is invariant under translations by such $Q$'s. This shows that in order to prove  Theorem \ref{elli} it is enough to show that $\psi$, viewed as an element of $B_2=\widehat{A}[T',T'']^{\widehat{\ }}$, is quasi linear, non-degenerate if $E$ is superordinary or ordinary defined over $\bZ_p$, and totally degenerate if $E$ is supersingular.
For  the following let us  recall from \ref{Fib} the ring
$$F^0B_2:=B_0+pB_1+p^2B_2.$$ Set also $S_0 := R[[T]]$ and utilize similar notation for $S_n$ as for $B_n$. 

\begin{lemma}\label{lem1}
The following  holds in  $B_2$:
$$
\psi  \in  \Omega^{p^2}\cdot (T')^{\phi}+\lambda_1\Omega^p\cdot T'+ F^0B_2.$$
\end{lemma}

\begin{proof}
Note that in $S_2$ we have:
$$T^{\phi^2}\in T^{p^2}+p(T')^p+p^2T''+p^2S_1$$
hence using 
$$ic_i\in R\ \ \ \text{for}\ \ \ i\geq 1$$
 and 
$$\frac{p^{j}}{j!}\in p^2R\ \ \text{for}\ \ p\geq 5,\ \ j\geq 2$$
 we get, for $i\geq 1$,
$$\begin{array}{rcl}
c_i^{\phi^2}(T^{\phi^2})^i & \in & c_i^{\phi^2}T^{p^2i}+ic_i^{\phi^2}T^{p^2(i-1)}(p(T')^p+p^2T''+p^2S_1)+\\
\ & \ & \ \\
\ & \ & +\sum_{j\geq 2} \left(\begin{array}{c} i\\ j\end{array}\right)c_i^{\phi^2}(p^jS_1+p^{j+1}S_2)\\
\ & \ & \ \\
\ & \subset & c_i^{\phi^2}T^{p^2i}+ic_i^{\phi^2}T^{p^2(i-1)}(p(T')^p+p^2T'')+p^2S_1+p^3S_2.
\end{array}$$
Similarly we get
$$c_i^{\phi}(T^{\phi})^i\in c_i^{\phi}T^{pi}+ipc_i^{\phi}T^{p(i-1)}T'+p^2S_1.$$
Set 
$$\psi_0:=\frac{1}{p}\{\sum c_i^{\phi^2}T^{p^2i}+\lambda_1\sum c_i^{\phi}T^{pi}+p\lambda_0\sum c_iT^i\}\in R[1/p][[T]].$$
So, by equation \ref{psi}, we get
$$\begin{array}{rcl}
\psi & \in & \psi_0+pS_1+p^2S_2+\\
\ & \ & \ \\
\ & \ & +(\sum ic_i^{\phi^2}T^{p^2(i-1)})(T')^p+\lambda_1(\sum ic_i^{\phi}T^{p(i-1)})T'+p(\sum ic_i^{\phi^2}T^{p^2(i-1)})T''\\
\ & \ & \ \\
\ & = & \Omega^{p^2} \cdot (T')^{\phi}+\lambda_1\Omega^p\cdot T' + \psi_0+pS_1+p^2S_2.
\end{array}$$
In particular $\psi_0\in S_0=R[[T]]$.
On the other hand, since $p^kR[[T]]\cap A=p^kA$ we immediately deduce that
$$(S_0+pS_1+p^2S_2)\cap B_2=B_0+pB_1+p^2B_2.$$

\end{proof}

{\it Proof of Theorem~\ref{elli}}. Assertion $1$ in the theorem holds by Lemma~\ref{lem1}, i.e., $\psi\in B_2$ is quasi-linear. The assertions concerning non-degeneracy/total degeneracy follow now from  Lemma \ref{contrast}. \qed

\

\begin{remark}
In this section we investigated the kernel of the $\d$-characters $\psi$ of elliptic curves with no Frobenius lifts. One can ask for an analysis of the kernel of $\psi$ for elliptic curves $R$ that are canonical lifts or for ${\mathbb G}_m$. The analysis in these cases is much simpler, the $\d$-characters have order $1$ rather than order $2$, and we leave it to the reader.\end{remark}

\section {Perfectoid spaces attached to $\d$-isogeny classes}

\subsection{$\d$-period map}
We recall some concepts from~\cite{Barcau, difmod, equations,local}.
The ring of {\em $\d$-modular functions} 
is
\[
    M^r:=R[a_4^{(\leq r)},a_6^{(\leq r)}, \Delta^{-1}]\h,
\]
where
$$\Delta:=-2^6a_4^3-2^43^3a_6^2,$$
 $a_4^{(\leq r)}$ is a tuple of variables,
$(a_4,a'_4,a''_4,\ldots ,a_4^{(r)})$, and $a_6^{(\leq r)}$ is defined similarly.
If $G \in M^0 \setminus pM^0$, define
\[
    M^r_{\{G\}} := M^r[G^{-1}]\h
    = R[a_4^{(\leq r)},a_6^{(\leq r)}, \Delta^{-1},G^{-1}]\h.
\]
An element of $M^r$ or $M^r_{\{G\}}$ is {\em defined over $\bZ_p$}
if it belongs to the analogously defined ring with $\bZ_p$ in place of $R$.
We have natural $p$-derivations $\d\colon M^r \to M^{r+1}$
and $\d\colon M^r_{\{G\}} \to M^{r+1}_{\{G\}}$.
Let $$j:-2^{12}3^3 a_4^3/\Delta,\ \ \ i:=2^63^3-j,\ \ \ t:=a_6/a_4,\ \ \ T:=a_6^2/a_4^3.$$
By \cite[Prop. 3.10]{difmod} we have
$$M^r_{\{a_4a_6\}}
    = R[j^{(\leq n)},j^{-1}, i^{-1}, t^{(\leq r)}, t^{-1}]\h.$$
If $w=\sum n_i \phi^i \in \Z[\phi]$, define $\deg w = \sum n_i$.
If moreover $\lambda \in R$,
define the symbol $\lambda^w:=\prod (\lambda^{\phi^i})^{n_i}$.
For $w \in \Z[\phi]$,
say that $f$ in $M^r$ or $M^r_{\{G\}}$ is of {\em weight} $w$
if
\begin{equation}
\label{nu shtiu} f(\lambda^4a_4,\lambda^6a_6,\d(\lambda^4a_4),\d(
\lambda^6a_6),\ldots )=\lambda^w f(a_4,a_6,a'_4,a'_6,\ldots ),
\end{equation}
for all $\lambda \in R$.
Let $M^r(w)$ be the set of $f \in M^r$ of weight $w$,
and define $M^r_{\{G\}}(w)$ similarly.
In~\cite{difmod}, elements of $M^r_{\{G\}}(w)$ were called
{\em $\d$-modular forms of weight $w$} (holomorphic outside $G=0$).

If $f \in M^r_{\{G\}}(w)$ and $E$ is an elliptic curve given by
$y^2=x^3+Ax+B$ with $A,B \in R$ and $G(A,B) \in R^{\times}$,
then define $f(A,B) \in R$ by making the substitutions
$a_4 \mapsto A$, $a_6 \mapsto B$, $a'_4 \mapsto \d A$, $a_6' \mapsto \d B$,
$a''_4 \mapsto \d^2 A$, and so on.

A from $f$ is called {\em isogeny covariant} if for
any isogeny $u$ of degree prime to $p$ from an elliptic
curve $y^2=x^3+A_1 x+B_1$ with $G(A,B) \in R^{\times}$ to an
elliptic curve $y^2=x^3+A_2 x+B_2$ with $G(A_2,B_2) \in R^{\times}$
that pulls back $dx/y$ to $dx/y$ we have
\[
    f(A_1,B_1) = \deg(u)^{-\deg(w)/2} f(A_2,B_2).
\]

By~\cite[Cor. 3.11]{difmod},
$M^r_{\{a_4a_6\}}(0)=R[j^{(\leq r)}, j^{-1}, i^{-1}]\h$.
More generally, if $m \in 2\Z$ and $g \in M^0(m)$,
define $\tilde{g}:=gt^{-m/2}$; then
\begin{equation}
\label{urssu}
    M^r_{\{a_4a_6g\}}(0)
    = R[j^{(\leq r)}, j^{-1}, i^{-1},\tilde{g}^{-1}]\h.
\end{equation}
Also define the open subscheme
$$X:={\mathbb A}^1_{a_4a_6g}:=\text{Spec}\  R[j,j^{-1},i^{-1}, \tilde{g}^{-1}]$$
of the modular curve ${\mathbb A}^1:= \text{Spec}\  R[j]$, i.e.,  the ``$j$-line".
The geometric points of $X$ are given by the values of $j$ different from $0,1728$ and from the zeros of $g$.
We have
\begin{equation}
\label{E:definition of b}
    T=-2^23^{-3}+2^8j^{-1},\ \ \ \ R[j,j^{-1},i^{-1}]=R[T,T^{-1},(4+27T)^{-1}],\end{equation}
so $T$ is an \'etale coordinate on 
$X={\mathbb A}^1_{a_4a_6g}$,
so
\begin{equation}
\label{mazel} B_r:=\cO(J^r(X))=\widehat{A}[T',\ldots ,T^{(r)}]\h,\ \ \ A:=\cO(X),
\end{equation}
where $T',\ldots ,T^{(r)}$ are new indeterminates.
Similarly, since $j$ is an \'{e}tale coordinate on ${\mathbb A}^1$,
 \eqref{urssu} yields
 \begin{equation}
 \label{foxx}
B_r=M^r_{\{a_4a_6g\}}(0).
 \end{equation}
There is a natural {\it  $\d$-Fourier expansion maps}
\begin{equation}
\label{catt} B_r \to S^r_{\infty}:=R((q))\h[q',...,q^{(r)}]\h.
\end{equation}
Let $E_4(q)$ and $E_6(q)$ be the normalized Eisenstein series of
weights $4$ and $6$, where normalized means constant
coefficient equal to $1$.
We have unique ring homomorphisms, also
referred to as {\it $\d$-Fourier expansion maps} 
\begin{align}
\label{sqrl}
    M^r &\to S^r_{\infty} \\
\notag
    g &\mapsto g_{\infty}=g(q,q',\ldots ,q^{(r)}),
\end{align}
sending $a_4$ and $a_6$ to
$-2^{-4}3^{-1}E_4(q)$ and $2^{-5}3^{-3}E_6(q)$, respectively,
and commuting with $\d$. The maps \ref{catt} and \ref{sqrl} are compatible in the obvious sense.
Also recall that there exists a unique $E_{p-1} \in M^0(p-1)$
such that $E_{p-1}(q)$ is the normalized Eisenstein series of weight $p-1$.

By \cite[Construction (4.1) and Corollary 7.26 ]{difmod},
there exists a unique $f^1 \in M^1(-1-\phi)$, defined over $\Z_p$,
such that
\begin{equation}
\label{maus}
    f^1(q,q')
    =\frac{1}{p} \log \frac{q^{\phi}}{q^p}
    := \sum_{n \geq 1}
        (-1)^{n-1}n^{-1} p^{n-1} \left( \frac{q'}{q^p} \right)^n
    \in R((q))\h[q']\h.
\end{equation}
As explained in \cite[pg. 126--129]{difmod}
$f^1$ is isogeny covariant and may
be interpreted as an
{\em arithmetic Kodaira-Spencer class}. By \cite[Lem. 4.4]{char} the elliptic curve
$E$ defined by $y^2=x^3+Ax+B$ is superordinary if and only if 
$$E_{p-1}(A,B)f^1(A,B)\not\equiv 0\ \ \ \text{mod}\ \ \ p.$$

Define now
\begin{equation}
\label{fffut}
    t^{\frac{\phi+1}{2}}
    := t^{\frac{p+1}{2}} \left( \frac{t^{\phi}}{t^p} \right)^{1/2}
    = t^{\frac{p+1}{2}} \sum_{j \geq 0}
            \binom{1/2}{j} p^j \left( \frac{\d t}{t^p} \right)^j
     \in M^1_{\{a_4a_6\}}(1+\phi),
\end{equation}
\begin{equation}
\label{varfi}
    \varphic :=f^1 \cdot t^{\frac{\phi+1}{2}}
        \in M^1_{\{a_4a_6\}}(0) \subset M^1_{\{a_4a_6 g\}}(0) =\cO(J^1(X))=B_1.
\end{equation}
As explained in \cite[pg. 588]{local}, using \cite[Thm. 1.3]{H} one gets 
an equality of the form:
\begin{equation}
\label{lolla} \varphic=\alpha T' +\beta,\ \ \ \beta=\beta_0+p \beta_1,
\end{equation}
for some $\beta_0 \in M^0_{\{a_4a_6\}}(0)$
and $\beta_1 \in M^1_{\{a_4a_6\}}(0)$, and
$$\alpha=cE_{p-1}\Delta^{-p}a_4^{\frac{7p-1}{2}}a_6^{-\frac{p-1}{2}}\in M^0_{\{a_4a_6\}}(0),$$
where $c\in R^{\times}$. 

Form this point on we assume $g:=E_{p-1}$; in this case we have $\alpha\in \widehat{A}^{\times}$, hence:

\begin{lemma}\label{fflat}
For $X={\mathbb A}^1_{a_4a_6E_{p-1}}$  we have that 
$$f^{\flat}\in \alpha T' +B_0+pB_1.$$
\end{lemma}

In particular
$f^{\flat}\in \cO(J^1(X))$ is quasi-linear
of order $(0,1)$ and non-degenerate.

\medskip

Now by \cite[Construction 3.2 and Thm. 5.1]{Barcau},
there exist unique $\d$-modular forms
$f^{\partial} \in M^1_{\{E_{p-1}\}}(\phi-1)$ and
$f_{\partial} \in M^1_{\{E_{p-1}\}}(1-\phi)$, defined over $\Z_p$, with
$\d$-Fourier expansions identically equal to $1$.
Moreover, these forms are isogeny covariant
and $f^{\partial} \cdot f_{\partial}=1$.
Furthermore, the reduction $\overline{f^{\partial}} \in \overline{M^1}$
equals the image of $E_{p-1}\in M_{p-1}$ in $\overline{M^1}$.
In particular
\begin{equation}
\label{fpar}
f_{\partial}\in E^{-1}_{p-1}+pB_1.
\end{equation}
For $\lambda \in R^{\times}$, we defined in \cite[pg. 590, Eq. 4.54]{local} the form
\medskip
\begin{equation}
\label{flam}
    f_{\lambda}  :=  (f^1)^{\phi}-\lambda f^1
f_{\partial}(f_{\partial})^{\phi} \in M^2_{\{E_{p-1}\}}(-\phi-\phi^2).
\end{equation}
\medskip
Since $f_1$ and $f^\partial$ are isogeny covariant,
so is $f_{\lambda}$.
Furthermore consider the series
\[t^{\frac{\phi^2+\phi}{2}}:=t^{\frac{p^2+p}{2}} \left(
\frac{t^{\phi}}{t^p} \right)^{1/2} \left( \frac{t^{\phi^2}}{t^{p^2}}
\right)^{1/2} \in M^2_{\{a_4a_6\}}(\phi+\phi^2),\] and
 define
\begin{equation}
\label{flats} f^{\flat}_{\lambda}  :=  f_{\lambda} \cdot
t^{\frac{\phi^2+\phi}{2}} 
\in M^2_{\{a_4a_6E_{p-1}\}}(0)=\cO(J^2(X))=B_2.
\end{equation}
\medskip

 Note that one may see $f^{\flat}_{\lambda}\in \cO(J^2(X))$ as a morphism 
 $$f^{\flat}_{\lambda}:J^2(X)\ra \widehat{{\mathbb A}^1}.$$
 One can introduce the following:
 
 \begin{definition}
  The $\d$-{\it period map} is the $\d$-morphism
 \begin{equation}\label{permap}
 \wp:J^2(X)\ra \widehat{{\mathbb A}^2}:=\text{Spf}\ R[x_0,x_1]^{\widehat{\ }}\end{equation}
 with components
 $$(
 t^{\frac{\phi^2+\phi}{2}}
  f^1
f_{\partial}(f_{\partial})^{\phi}, t^{\frac{\phi^2+\phi}{2}}(f^1)^{\phi}).$$
The induced morphism of perfectoid spaces in ${\mathcal P}_{\Phi}$,
\begin{equation}
\label{perfpermap}
P^{\infty}(\wp):P^{\infty}(X)\ra P^{\infty}({\mathbb A}^2),
\end{equation}
will still be called the $\d$-{\it period map}.
 \end{definition}
 
 \begin{remark}
 The morphism $\wp$ induces a morphism
 which we call the {\it projectivized $\d$-period map},
 \begin{equation}\label{projper}
 \tilde{\wp}:J^2(X)\backslash \wp^{-1}(0)\stackrel{\wp}{\ra} \widehat{{\mathbb A}^2}\backslash \{0\}\stackrel{\text{can}}{\ra} \widehat{{\mathbb P}^1}:=(\text{Proj}\ R[x_0,x_1])^{\widehat{\ }}.
 \end{equation}
Clearly, if $L_{\lambda}\subset \widehat{{\mathbb A}^2}$ is the line defined by
$x_1-\lambda x_0$ then we have an equality of formal schemes
$$(f^{\flat}_{\lambda})^{-1}(0)=\wp^{-1}(L_{\lambda})=\tilde{\wp}^{-1}(1:\lambda).$$

However, we cannot apply our theory to attach a map of perfectoid spaces to the projectivized $\d$-period map \ref{projper} because  $\tilde{\wp}$ is not defined on all of $J^2(X)$. However, a variant of our theory can be developed to define such a map of perfectoid spaces, $P^{\infty}(\tilde{\wp})$ and once done, one is then tempted to expect that $P^{\infty}(\tilde{\wp})$ possesses no section in the category of perfectoid spaces (or at least in the category ${\mathcal P}_{\Phi}$). We will not pursue this issue here.\end{remark}
 
 \begin{remark}
 Our definition of the $\d$-period map \ref{permap} may seem somewhat ad hoc., the motivation for this comes from \cite[Prop. 8.75]{equations}. However, that this map ``generates", in a precise sense, the ``$\d$-ring of all $\d$-morphisms satisfying a certain {\it isogeny covariance} property". So the map \ref{permap}
 is actually a ``canonical" $\d$-map.
 \end{remark}

Recall from \cite[Lem. 4.55]{local} the following remarkable property of the forms $f_{\lambda}$:

\begin{lemma}
\label{9438hfb}
Let $E_1$ be an elliptic curve $y^2=x^3+A_1x+B_1$ over $R$
with ordinary reduction.  Then
\begin{enumerate}
\item
There exists $\lambda \in R^{\times}$ such that $f_{\lambda}(A_1,B_1)=0$.
\item
If $\lambda$ is as in (1) and there is an isogeny of
degree prime to $p$ between $E_1$ and an elliptic curve $E_2$
over $R$ given by $y^2=x^3+A_2x+B_2$, then $f_{\lambda}(A_2,B_2)=0$.
\end{enumerate}
\end{lemma}

If $E_1$ in the Lemma above is superordinary then $\lambda$ in assertion 1 is, of course, given by
$$\lambda:=\lambda(E_1):=\frac{(f^1)^{\phi}f^{\partial}(f^{\partial})^{\phi}}{f^1}(A_1,B_1).$$
If in addition $E_2$ is as in the Lemma then $E_2$ is also superordinary and 
$$\lambda(E_1)=\lambda(E_2).$$

Motivated by the above Lemma we make the following definition. 

\begin{definition}Let $X={\mathbb A}^1_{a_4a_6E_{p-1}}$
and  $\lambda\in R^{\times}$.
The  
{\it $\d$-isogeny class} attached to $\lambda$ is the arithmetic differential equation
 $Y_2(\lambda)\subset J^2(X)$  defined by the ideal generated by $f^{\flat}_{\lambda}$; equivalently $Y_2(\lambda):=\wp^{-1}(L_{\lambda})=\tilde{\wp}^{-1}(1:\lambda)$, where $\wp$ is the $\d$-period map, $L_{\lambda}\subset \widehat{{\mathbb A}^2}$ is the line $x_1-\lambda x_0=0$, and $\tilde{\wp}$ is the projectivized $\d$-period map.
 If $E$ is a superordinary elliptic curve curve and $Q\in X(R)$ is the point corresponding to it
 then we set $Y_2(Q):=Y_2(\lambda(E))$. \end{definition}
 
  Note that if $Q'\in X(R)$ is another point corresponding to an elliptic curve $E'$ that has possesses an isogeny to $E$ of degree prime to $p$ then $Y_2(Q)=Y_2(Q')$.

\subsection{Quasi-linearity of $Y_2(\lambda)$ and a perfectoid consequence}\label{sec:QL} We are now in a position to apply our theory to $\d$-isogeny classes. We also prove Corollary~\ref{ppp} which provided the contradiction in proof of Theorem~\ref{sandoval}.

\begin{theorem}\label{oop}
Let $X={\mathbb A}^1_{a_4a_6E_{p-1}}$
and  $\lambda\in R^{\times}$. The $\d$-isogeny class  $Y_2(\lambda)\subset J^2(X)$ attached to $\lambda$  is quasi-linear of order $(1,1)$ and non-degenerate. In particular for any 
 superordinary $Q\in X(R)$ the  $\d$-isogeny class  $Y_2(Q)\subset J^2(X)$  is quasi-linear of order $(1,1)$ and non-degenerate. 
\end{theorem}
\begin{proof}
By Lemma \ref{fflat} we have
$$f^{\flat}\in \alpha T' +B_0+pB_1,\ \ \ \ \ \ \ (f^{\flat})^{\phi} \in \alpha^p (T')^{\phi} +B_0+pB_1+p^2B_2.$$
On the other hand we have
$$f^{\flat}_{\lambda}=
(f^{\flat})^{\phi}-\lambda f^{\flat} f_{\partial}(f_{\partial})^{\phi}t^{\frac{\phi^2-1}{2}}$$
where
$$t^{\frac{\phi^2-1}{2}}:=t^{\frac{p^2-1}{2}}\cdot (t^{\phi^2}/t^{p^2})^{1/2}\in 
t^{\frac{p^2-1}{2}}+pB_1.
$$
Moreover, by \ref{fpar}, we have
$$f_{\partial}\in E_{p-1}^{-1}+pB_1,\ \ \ (f_{\partial})^{\phi}\in E_{p-1}^{-p}+pB_1+p^2B_2.$$
Combining the above formulas we get
$$
\begin{array}{rcl}
f^{\flat}_{\lambda} & \in & (\alpha^p (T')^{\phi} +B_0+pB_1+p^2B_2) \\
\ & \ & \ \\
\ & \ & -\lambda(\alpha T' +B_0+pB_1)(E_{p-1}^{-1}+pB_1)(E_{p-1}^{-p}+pB_1+p^2B_2)
(t^{\frac{p^2-1}{2}}+pB_1)\\
\ & \ & \ \\
\ & \subset & \alpha^p (T')^{\phi}-\lambda \alpha E_{p-1}^{-p-1}t^{\frac{p^2-1}{2}}\cdot T'
+B_0+pB_1+p^2B_2.
\end{array}
$$
This ends the proof because $\alpha\in \widehat{A}^{\times}$ and $E_{p-1}^{-p-1}t^{\frac{p^2-1}{2}}\in \widehat{A}^{\times}$.
\end{proof}

We finally address the a for mentioned corollary. As usual we denote by $(Y_n(\lambda))_{n\geq 0}$ the prolongation of $Y_2(\lambda)$. Set also
$C_n(\lambda)=\cO(Y_n(\lambda))$, and 
$$C_{\infty}(\lambda):=\varinjlim C_n(\lambda).$$
By Theorem \ref{oop} and Corollary \ref{argue} we get.

\begin{corollary}\label{ppp}
The $K$-algebra $P:=P(\widehat{C_{\infty}(\lambda)},\phi)$ is perfectoid and its relative Frobenius lift $\Phi$ is invertible. Moreover the ring
$P^\circ/{\mathbb K}^{\circ \circ}P^\circ$
 is the perfection of an integral ind-\'{e}tale algebra 
  over $\overline{A}=\cO(\overline{X})$.

  In addition, for $P^{\infty}(Y_2(Q)):=\Spa(P,P^\circ)$, there is a natural closed immersion $P^{\infty}(Y_2(Q))\ra P^{\infty}(X)$ compatible with the relative Frobenius lifts. 
  \end{corollary}
   \bigskip

The scheme $P^{\infty}(Y_2(Q))_k$ is, of course, the reduction mod ${\mathbb K}^{\circ \circ}$ of $P^{\infty}(Y_2(Q))$ with respect to the canonical principal cover of the affine scheme $X$; it is the perfection of a profinite pro-\'{e}tale cover of $\overline{X}$.
  Theorem \ref{boi} follows now from Corollary \ref{ppp} and from our construction. In particular, one 
simply needs to take $P^Q(X)$ in Theorem \ref{hoi} to be equal to $P^{\infty}(Y_2(Q))$ in Corollary \ref{ppp}.


\begin{thebibliography}{AA}


\bibitem{Barcau} M. Barcau, {\it Isogeny covariant differential
modular forms and the space of elliptic curves up to isogeny},
Compositio Math. 137 (2003), 237-273.


\bibitem{bhatt} B. Bhatt, {\it Lecture notes for a class on perfectoid spaces}, 2017.

\bibitem{BS18} B. Bhatt, P. Scholze, {\it Prisms and prismatic cohomology}, preprint, arxiv.org:1905.08229.

\bibitem{char} A. Buium,  {\it Differential characters of
abelian varieties over $p$-adic fields},  Invent. Math. 122, 2
(1995), 309-340.


\bibitem{frob} A. Buium, {\it Differential characters and
 characteristic polynomial of Frobenius},
 Crelle J., 485 (1997), 209-219.
 
 \bibitem{pjets} A. Buium, {\it Geometry of p-jets},  Duke Math. J., 82, 2, (1996), 349-367.


 
 
\bibitem{difmod} A. Buium, {\it Differential modular forms},
Crelle J. 520 (2000), 95--167.



\bibitem{equations} A. Buium, {\it Arithmetic Differential Equations}, Math. Surveys and Monographs, AMS, 2005.


\bibitem{pfin1} A. Buium, {\it  $p$-jets of finite algebras, I: $p$-divisible groups}, Documenta Math. 18 (2013) 943-969.

\bibitem{B15} A. Buium, {\it  Differential modular forms attached to newforms mod $p$}, Journal of Number Theory, 155, (2015), 111--128. 


\bibitem{local} A. Buium, B. Poonen, {\it Independence of points on elliptic
curves arising from special
 points on modular and Shimura curves, II: local results},
 Compositio Math., 145 (2009), 566-602.
 
 
 
\bibitem{BYM} 
A. Buium, Yu. I. Manin, {\it Arithmetic differential equations of Painlev\'{e} VI type},  
 in: Arithmetic and Geometry, London Mathematical Society Lecture
Note Series: 420,  L. Dieulefait, G. Faltings, D. R. Heath-Brown, Yu. V. Manin,
B. Z. Moroz and J.-P. Wintenberger (eds),  Cambridge University Press, 2015, pp. 114-138.



\bibitem{BS11} A. Buium, A. Saha {\it Differential Overconvergence}, Algebraic methods in dynamical systems; volume dedicated to Michael Singer’s 60th birthday, Banach Center Publications, Vol 94, 99-129 (2011).
 

\bibitem{greenberg}
 M.Greenberg, {\it Schemata over local rings},
 Annals of Math. 73 (1961), 624-648.

 
 
\bibitem{H}  Hurlburt, C., {\it Isogeny covariant
differential modular forms modulo p}, Compositio Math., {\bf 128} (2001),
no.1, 17-34.


\bibitem{Jo} A. Joyal, \textit{$\d-$anneaux et vecteurs de Witt},
C.R. Acad. Sci. Canada,
Vol. VII, No. 3, (1985), 177-182.




\bibitem{manin63} Yu. I. Manin, {\it Rational points on algebraic curves over function fields}, Izv. Acad. Nauk USSR, 27 (1963), 1395-1440.


\bibitem{manin} Yu. I. Manin, {\it Numbers as functions},  arXiv:1312.5160.




\bibitem{manpain} Yu. I. Manin. {\it Sixth Painlev\'e equation, universal elliptic curve,
and mirror of ${\mathbb P}^2$.}  In: geometry of Differential
Equations, ed. by A.~Khovanskii, A.~Varchenko, V.~Vassiliev.
Amer. Math. Soc.
Transl. (2), vol. 186  (1998), 131--151. Preprint alg--geom/9605010.


\bibitem{matsumura} H. Matsumura, {\it Commutative Ring Theory}, Cambridge Studies in Advanced
Mathematics, 8. Cambridge University Press, Cambridge, 1989. xiv+320 pp.

\bibitem{scholze} P. Scholze, {\it Perfectoid spaces},  Publ. Math. IHES, 116 (2012), 245-313.

\end{thebibliography}
\end{document}